\documentclass{article}
\usepackage[utf8]{inputenc}
\usepackage[T1]{fontenc}
\usepackage{amssymb,amsthm,amsmath,mathtools,amsrefs,MnSymbol}
\usepackage{amscd,amstext,units,amsfonts,amsbsy}
\usepackage[utf8]{inputenc}
\usepackage[colorlinks=true]{hyperref}
\usepackage{color,xcolor}
\usepackage{url}
\usepackage{comment}
\usepackage{tikz-cd}

\usepackage[toc,page]{appendix}
\usepackage{tikz-cd}
\usepackage{graphicx}
\graphicspath{ {images/} }

\newcommand{\bp}{\begin{prop}}
\newcommand{\ep}{\end{prop}}
\newcommand{\bd}{\begin{definicion}}
\newcommand{\ed}{\end{definicion}}
\newcommand{\bl}{\begin{lema}}
\newcommand{\el}{\end{lema}}
\newcommand{\bh}{\begin{hecho}}
\newcommand{\eh}{\end{hecho}}
\newcommand{\bpreg}{\begin{preg}}
\newcommand{\epreg}{\end{preg}}
\newcommand{\bo}{\begin{obs}}
\newcommand{\eo}{\end{obs}}
\newcommand{\bcon}{\begin{conj}}
\newcommand{\econ}{\end{conj}}
\newcommand{\brmk}{\begin{rmk}}
\newcommand{\ermk}{\end{rmk}}
\newcommand{\bc}{\begin{corol}}
\newcommand{\ec}{\end{corol}}
\newcommand{\bconst}{\begin{const}}
\newcommand{\econst}{\end{const}}

\newcommand{\bitem}{\begin{itemize}}
\newcommand{\eitem}{\end{itemize}}
\newcommand{\bt}{\begin{teor}}
\newcommand{\et}{\end{teor}}
\newcommand{\be}{\begin{ejem}}
\newcommand{\ee}{\end{ejem}}
\newcommand{\bnot}{\begin{nota}}
\newcommand{\enot}{\end{nota}}
\newcommand{\trdeg}{\operatorname{trdeg}}
\newcommand{\cl}{\operatorname{cl}}
\newcommand{\Vect}{\operatorname{Vect}}
\newcommand{\bdn}{\operatorname{bdn}}

\newtheorem*{conj}{Conjecture}
\newtheorem{claim}{Claim}

\newtheorem{theorem}{Theorem}

\numberwithin{theorem}{section}
\newtheorem{corollary}[theorem]{Corollary}
\newtheorem{definition}[theorem]{Definition}
\newtheorem{lemma}[theorem]{Lemma}
\newtheorem{notation}[theorem]{Notation}

\newtheorem{proposition}[theorem]{Proposition}

\newtheorem{fact}[theorem]{Fact}
\newtheorem{remark}[theorem]{Remark}
\usepackage{fancyvrb }
\newcommand{\forkindep}[1][]{%
  \mathrel{
    \mathop{
      \vcenter{
        \hbox{\oalign{\noalign{\kern-.3ex}\hfil$\vert$\hfil\cr
              \noalign{\kern-.7ex}
              $\smile$\cr\noalign{\kern-.3ex}}}
      }
    }\displaylimits_{#1}
  }
}

\numberwithin{equation}{section}

\def\Ind#1#2{#1\setbox0=\hbox{$#1x$}\kern\wd0\hbox to 0pt{\hss$#1\mid$\hss}
\lower.9\ht0\hbox to 0pt{\hss$#1\smile$\hss}\kern\wd0}

\def\ind{\mathop{\mathpalette\Ind{}}}

\def\notind#1#2{#1\setbox0=\hbox{$#1x$}\kern\wd0
\hbox to 0pt{\mathchardef\nn=12854\hss$#1\nn$\kern1.4\wd0\hss}
\hbox to 0pt{\hss$#1\mid$\hss}\lower.9\ht0 \hbox to 0pt{\hss$#1\smile$\hss}\kern\wd0}

\newcommand{\tp}{\operatorname{tp}}
\newcommand{\dcl}{\operatorname{dcl}}
\newcommand{\acl}{\operatorname{acl}}

\newcommand{\val}{\operatorname{val}}
\newcommand{\res}{\operatorname{res}}

\newcommand{\ac}{\operatorname{ac}}
\newcommand{\rv}{\operatorname{rv}}
\newcommand{\kInt}{\operatorname{kInt}}
\newcommand{\pp}{\operatorname{p.p.-type}}
\newcommand{\id}{\operatorname{id}}

\title{Residue field domination in henselian valued fields of equicharacteristic zero }
\author{Mariana Vicar\'ia}
\begin{document}
    \maketitle
\begin{abstract}
In this paper we investigate domination results in an Ax-Kochen/ Ershov style for henselian valued fields of equicharacteristic zero for elements in the home sort.
\end{abstract}
\newpage
  \tableofcontents 
\newpage
\section{Introduction}
The model theory of henselian valued fields has been a major topic of study during the last century. It was initiated by Robinson's model completeness results of algebraically closed valued fields  in \cite{robinson}.  Remarkable work has been achieved by Haskell, Hrushovski and Macpherson to understand the model theory of algebraically closed valued fields. In a sequence of papers \cite{HHM} and  \cite{HHM2}, they developed the notion of stable domination, that essentially established how an algebraically closed valued field is governed by its stable part. \\

In further work  Ealy, Haskell and Ma\v{r}\'{i}cov\'{a}  present in  \cite{RCVF} , introduced an abstract form the notion of domination present in \cite{HHM}, 
Let $T$ be a complete first order theory and let  $S$ and $\Gamma$ be stably embedded sorts, and $C \subseteq A,B$ be sets of parameters in the monster model $\mathfrak{C}$. 
\begin{definition}
\begin{enumerate}
    \item  the type $\tp(A/C)$ is said to be \emph{dominated by the sort $S$}, if whenever $S(B)$ is \emph{independent} from $S(A)$ over $S(C)$ then that $ \tp(A/ C S(B)) \vdash \tp(A/ CB)$.
    
    \item the type $\tp(A/ C)$ is said to be dominated by the sort $S$ over $\Gamma$ if the type $\tp(A/ C \Gamma(A))$ is dominated by the sort $S$. 
\end{enumerate}
 \end{definition}
In  \cite{RCVF} domination results for the setting of real closed convexly valued fields are proved, which suggests that the presence of a stable part of the structure is not fundamental to achieve domination results and indicates that the right notion should be residue field domination or domination by the internal sorts to the residue field in broader classes of henselian valued fields. Later in \cite{Haskell}, Haskell, Ealy and Simon generalized the residue field domination results for the theory of henselian valued fields of equicharacteristic zero with bounded Galois group. In their work, it becomes clear that the key ingredients to obtain domination results are the existence of separated basis and a relative quantifier elimination statement. \\

Our main motivation arises from the natural question of how much further a notion of residue field domination could be extended to broader classes of valued fields to gain a deeper model theoretic insight of henselian valued field.\\
Stable domination has played a fundamental role in understanding the model theory of algebraically closed valued fields, and more precisely it has served as a bridge to lift ideas from stability theory to the setting of valued fields. For example, Hrushovski and Rideau-Kikuchi in \cite{SH}, have shown that for any definable abelian group $A$ in a model of  $ACVF$ we can find a definable group $\Lambda \subseteq \Gamma^{n}$, where $\Gamma$ is the value group, and a definable homomorphism $\lambda: A \rightarrow \Lambda$, such that $H:=ker(\lambda)$ is limit stably dominated [see \cite[Definition 5.6]{SH}].\\

 In this paper we investigate domination results for henselian valued fields of equicharacteristic zero. The general strategy to show domination of a type $\tp(A/C)$ by a sort $S$ is taking a partial elementary map $\sigma$ witnessing $\tp(A/ C S(B))= \tp(A'/ CS(B))$ and finding an automorphism $\hat{\sigma}$ of the monster model that extends $\sigma$ and fixes $CB$. For each of the statements, we specify precisely which notion of \emph{independence} is required to extend the isomorphism, in fact not the entire power of forking independence is needed. \\
 
It is still an open question to find a reasonable language in which a henselian valued field of equicharacteristic zero eliminates imaginaries. Some positive results have been obtained in certain classes of henselian valued fields of equicharacteristic zero, see for example \cite{RH} and \cite{multivaluedfields}.\\
Therefore, we start by studying domination results for henselian valued fields of equicharacteristic zero for elements in the home-sort. A valued field $(K,\mathcal{O})$ where $\mathcal{O}$ is its valuation ring and $\mathcal{M}$ the maximal ideal, can be considered in several different languages. For instance, a valued field can be seen as a three sorted structure in the language $\mathcal{L}_{\val}$ [see Definition \ref{1}],
 where the first two sorts are equipped with the language of fields while the third one is provided with the language of ordered abelian groups $\mathcal{L}_{OAG}=\{ 0, <, +, -\}$ extended by a constant $\infty$. We interpret the first sort as the main field, the second one as the residue field and the third one as the value group $\Gamma$ and $\infty$, where $\gamma< \infty$ for all $\gamma \in \Gamma$ and $\gamma+\infty= \infty+\gamma=\infty$ for all 
 $\gamma \in \Gamma \cup \{ \infty \}$.\\
A natural extension of $\mathcal{L}_{\val}$ is the language where an angular component map is added, we denote this extension by  $\mathcal{L}_{\ac}$ [see Definition \ref{2}]. \\
% Let $(K, \mathcal{O})$ be a henselian valued field of equicharacteristic zero, where $\mathcal{O}$ is the valuation ring and $\mathcal{M}$ is the maximal ideal of $\mathcal{O}$.  Any henselian valued field can be seen as a three sorted structure $(K, \mathbf{k}, \Gamma)$ where the first two sorts are equipped with the language of fields while the third one is provided with the language of ordered abelian groups $\mathcal{L}_{OAG}=\{ 0, <, +, -\}$ extended by a constant $\infty$. We add a function symbol $v: K \rightarrow \Gamma \cup \{ \infty\}$ by extending the valuation to a monoid morphism sending $v(0)=\infty$. We add as well a function symbol $\res: K \rightarrow \mathbf{k}$, which sends an element $a \in \mathcal{O}$ to its residue class $\res(a)=a+\mathcal{M}$, while $res(a):= 0$ for any element $a \in K \backslash \mathcal{O}$. We refer to this language as the $\mathcal{L}_{val}$-language.\\
 
 In \cite{distal}, Aschenbrenner, Chernikov, Gehret and Ziegler introduced a multi-sorted language $\mathcal{L}$ extending $\mathcal{L}_{\val}$ in which one obtains elimination of field quantifiers for any henselian valued field of equicharacteristic zero.
 In this extension,  we expand the structure $(\mathbf{k}, \Gamma \cup \{\infty \})$ by adding a new sort $\mathbf{k}/(\mathbf{k}^{\times})^{n}$ for every $n \geq 2$ together with the natural surjections $\pi_{n}: \mathbf{k} \rightarrow \mathbf{k}/ (\mathbf{k}^{\times})^{n}$. A precise description of the language $\mathcal{L}$ is given in Definition \ref{3}. The multi-sorted structure $(\mathbf{k}/(\mathbf{k}^{\times})^{n} \ | \ n \in \mathbb{N})$ will play a fundamental role in the domination results, so we refer to them as the \emph{power residue sorts}. We identify $\mathbf{k}/\mathbf{k}^{0}$ with the residue field $\mathbf{k}$. \\

This paper is organized as follows:
 \begin{enumerate}
\item Section \ref{preliminaries}: we briefly summarize the relative quantifier elimination already known for henselian valued fields in equicharacteristic zero. 
 \item Section \ref{dominationvalres}: we prove that  over a maximal model, a valued field is dominated  by the value group and the power residue sorts in the language $\mathcal{L}$ introduced by Aschenbrenner, Chernikov, Gehret and Ziegler. We show as well that over a maximal field, a valued field is dominated by the residue field and the value group sort in the $\mathcal{L}_{\ac}$-language.  
  \item Section \ref{fork}: we prove that over a maximal model forking is determined by Shelah's imaginary expansion of the power residue sorts (residue sort) and Shelah's imaginary expansion of the value group in the language $\mathcal{L}$ (or $\mathcal{L}_{\ac}$). We assume the theory of the residue field to be $NTP_{2}$. 
 \item Section \ref{domint}: we show that over a maximal model a valued field is dominated by the sorts internal to the residue field over the value group in the language $\mathcal{L}$. 

% \item Section \ref{resolutions}: We introduce the notion of \emph{weak opacity} to construct resolutions for the class of henselian valued fields with residue field algebraically closed and whose value group is $dp$-minimal, regular and satisfies the following property: every definably closed set is a model of the theory. We conclude this section by lifting the domination results to imaginary elements. 
 \end{enumerate}

  We follow ideas present in \cite{HHM}, and some results are present in their proofs but not in an easy quotable form. We include the statements and their proofs for sake of completeness.\\

\textbf{Aknowledgements:} This document was written as part of the author's  PhD thesis under  Pierre Simon and Thomas Scanlon. The author would like to express her gratitude to both of them for many insightful conversations. The author would like as well to thank Clifton Ealy, for answering some questions in a kind and organized way when the author was learning many of the contents related to this work. 
\newpage
\tableofcontents
\section{Preliminaries}\label{preliminaries}
\subsection{Finer valuations, places and separated bases}\label{refi}
To gain our final statement of domination of the internal sorts to the residue field, we will need to construct a refined valuation induced by the composition of some places. In this subsection, we recall some basics facts about refinements and places. We refer the reader interested in further details to \cite[Section 2.3]{EP}. \\
\begin{definition} Let $\mathcal{O}$ be a valuation ring of $K$ and $\mathcal{O}'$ be an overring of $\mathcal{O}$, and hence a valuation ring of $K$.  Then, we say that $\mathcal{O}$ is a \emph{coarsening} of $\mathcal{O}'$ and $\mathcal{O}'$ is a \emph{refinement} of $\mathcal{O}$. 
\end{definition}
Let $\mathcal{O}$ be a fixed valuation ring of $K$ and $\mathcal{O}'$ be an overring of $\mathcal{O}$. We have $\mathcal{M}'\subseteq \mathcal{M}$, where $\mathcal{M}'$ and $\mathcal{M}$ denote the maximal ideals of $\mathcal{O}'$ and $\mathcal{O}$ respectively. Since $\mathcal{M}'$ is a prime ideal in $\mathcal{O}'$, then it is also a prime ideal of $\mathcal{O}$. Moreover, localizing $\mathcal{O}$ at $\mathcal{M}'$ we can recover $\mathcal{O}'$, in fact $\mathcal{O}'= \mathcal{O}_{\mathcal{M'}}$. \\

The following is \cite[Lemma 2.3.1]{EP}.
\begin{lemma}\label{covexprime} Let $\mathcal{O}$ be a non trivial valuation ring in $K$ corresponding to the valuation $v:K \twoheadrightarrow \Gamma \cup \{ \infty\}$. Then there is a $1$-to-$1$ correspondence of the convex subgroups $\Delta$ of $\Gamma$ with the prime ideals $p$ of $\mathcal{O}$, and hence with the overrings $\mathcal{O}_{p}$. This correspondence is given by:
\begin{align*}
\Delta &\rightarrow p_{\Delta}=\{ x \in K \ | \ v(x) > \delta \ \text{for all} \ \delta \in \Delta \}\\
p &\rightarrow \Delta_{p}=\{ \gamma \in \Gamma \ | \ \gamma< v(x) \ \text{and} -\gamma < v(x) \ \text{for all} \ x\in p\}. 
\end{align*}
\end{lemma}
 Let $\mathcal{O}$ be a valuation ring of $K$ and $v:K \rightarrow \Gamma \cup \{ \infty\}$ the corresponding valuation. Let $p$ be a prime ideal with corresponding convex subgroup $\Delta_{p}$ in $\Gamma$ and  $\mathcal{O}_{p}$ the refinement of $\mathcal{O}$. There is a group homomorphism:
 \begin{center}
$\phi: \begin{cases}
K^{\times}/\mathcal{O}^{\times} &\rightarrow K^{\times}/ \mathcal{O}_{p}^{\times}\\
x\mathcal{O}^{\times} &\mapsto x \mathcal{O}_{p}^{\times}. 
\end{cases}$
\end{center}
whose kernel is $\Delta_{p}\cong \mathcal{O}^{\times}_{p}/ \mathcal{O}^{\times}$. The valuation $v_{p}$ induced by $\mathcal{O}_{p}$ is therefore obtain from $v:=K \rightarrow \Gamma \cup \{ \infty\}$ simply by taking the quotient of 
$\Gamma$ by the convex subgroup $\Delta_{p}$. % We denote as  $\Gamma_{p}$ the value group of the valuation $v_{p}$ and we let $\pi_{\Delta}:= \Gamma \rightarrow \Gamma/\Delta$ the canonical projection map. %Therefore the  following diagram commutes: 
%\begin{center}
%\begin{tikzcd}\label{diag}
%K^{\times}\arrow[r, "v"] \arrow[d, "v_{p }"]
%& \Gamma \arrow[d, "\pi_{\Delta}" ] \\ \Gamma_{p}
%&\Gamma/\Delta.   \arrow[l, "\phi" ]\end{tikzcd}
%\end{center}
 
\begin{definition} Let $M$ and $C$ be valued fields. Let $m_{1},\dots,m_{k}$ be elements of $M$, we write $\Vect_{C}(m_{1},\dots,m_{k})$ for the $C$-vector space generated by $\{m_{1},\dots,m_{k}\}$. We say that $\{m_{1},\dots,m_{k}\}$ is \emph{separated over $C$} if for all $c_{1},\dots,c_{k} \in C$, we have:
\begin{align*}
v \big( \sum_{i=1}^{k} c_{i}m_{i} \big)=\min\{ v(c_{i}m_{i}) \ | \ 1 \leq i \leq k\}.
\end{align*}
In particular, it is a basis for $\Vect_{C}(m_{1},\dots,m_{k})$. In general, we say that $M$ has the \emph{separated basis property} over $C$ if every finite dimensional $C$-subspace $\Vect_{C}(m_{1}, \dots, m_{k})$ where $\{m_{1},\dots,m_{k}\} \subseteq M$ has a separated basis.\\
 If $ C \subseteq M$, a separated basis is said to be \emph{good} if in addition for all $1\leq i,j \leq k$, either $v(m_{i})=v(m_{j})$ or $v(m_{i})-v(m_{j}) \notin \Gamma_{C}$, and we say that $M$ has the \emph{good separated basis property} over $C$ if every finite dimensional $C$-subspace $\Vect_{C}(m_{1}, \dots, m_{k})$ where $\{m_{1},\dots,m_{k}\} \subseteq M$ has a good separated basis.
\end{definition}
The following is a folklore fact, details can be found for example in \cite[Lemma 12.2]{HHM}. 
\begin{fact}\label{maximal} Suppose $C$ is maximally complete. Then every valued field extension $M$ of $C$ has the good separated basis property. 
\end{fact}

\begin{definition} Let $K$ and $L$ be fields. A map $\phi: K \rightarrow L \cup \{ +\infty\}$ is a \emph{place over $K$} if for any $x,y \in K$:
\begin{itemize}
\item $\phi(x+y)=\phi(x)+\phi(y)$,
\item $\phi(x\cdot y)= \phi(x) \cdot \phi(y)$,
\item $\phi(1)=1$. 
\end{itemize}
Here for all $a \in L$, the following operations are defined $a+\infty=\infty+a=\infty$, and $a \cdot \infty=\infty \cdot a=\infty \cdot \infty=\infty$. While the operations $\infty+\infty, 0 \cdot \infty, \infty \cdot 0$ are not. 
\end{definition}
The following proposition states the correspondence between places and valuations over a field. This a folklore fact, for example see \cite[ Exercise 2.5.4]{EP}.
\begin{proposition}\label{placeval} Let $K$ and $L$ be fields and $\phi: K \rightarrow L \cup \{ \infty\}$ be a place over $K$. Then $\mathcal{O}=\phi^{-1}(L)$ is a valuation ring of $K$ whose maximal ideal is $\mathcal{M}=\phi^{-1}(\{0\})$ and its residue field is $\phi(K)\backslash \{\infty\}$. Moreover, given a valuation ring $\mathcal{O}$ of $K$ whose maximal ideal is $\mathcal{M}$ the map:
\begin{align*}
\phi: K &\rightarrow \mathcal{O}/\mathcal{M} \cup \{ +\infty \}\\
&\begin{cases}
x \rightarrow x+\mathcal{M} \ \text{if $x \in \mathcal{O}$,}\\
x \rightarrow \infty \ \text{if $x \in K \backslash \mathcal{O}$.}\\
\end{cases}
\end{align*}
\text{is a place of $K$.}
\end{proposition}
\begin{notation} Let $\Gamma$ be an ordered abelian group, and let $\gamma, \delta \in \Gamma$. We write $\gamma \ll \delta$  to indicate that $n\gamma< \delta$ for all $n \in \mathbb{N}$. 
\end{notation}
We conclude this subsection stating a lemma that we will need to obtain a domination result by the internal sorts of the residue field over the value group. 
\begin{lemma}\label{inf} Let $v: L\rightarrow \Gamma$ be a valuation on a field $L$. Let $p: L \rightarrow \res(L) \cup \{\infty \}$ be the place corresponding to the valuation $v$ and $F$ be a subfield of $\res(L)$ and $p':\res(L) \rightarrow F$ be a place which is the identity on $F$. Let $p^{*}$ the composition of places $p' \circ p: L \rightarrow F$, and $v^{*}:L \rightarrow \Gamma^{*}$ the induced valuation. Suppose that $a \in L$ with $p(a) \subseteq \res(L) \backslash \{ 0\}$ and $p^{*}(a)=0$. Then:
\begin{enumerate}
\item if $\Delta=\{ v^{*}(x), -v^{*}(x) \ | \ x \in L,\ p(x) \notin \{ \infty, 0\} \ , p^{*}(x)=0\} \cup \{ 0_{\Gamma^{*}}\}$. Then $\Delta$ is a convex subgroup of $\Gamma^{*}$ and there is an isomorphism of ordered abelian groups $g: \Gamma^{*}/\Delta \rightarrow \Gamma$ such that $g\circ v^{*}=v$,
\item if $b \in L$ with $v(b)>0$, then $0 <v^{*}(a) \ll v^{*}(b)$,
\item let $M \subseteq L$ be a subfield. If $(r_{1},\dots,r_{n})$ is a separated basis of the $M$-vector subspace $\Vect_{M}(r_{1},\dots,r_{n})$ according to the valuation $v^{*}$ then it is also a separated basis according to the valuation $v$. Furthermore, if $\displaystyle{v^{*} \big( \sum_{i=1}^{n}r_{i}m_{i}\big)= \min\{ v^{*}(r_{i}m_{i}) \ | \ i \leq n \} = v^{*}(r_{j}m_{j})}$, then  $\displaystyle{v \big( \sum_{i=1}^{n}r_{i}m_{i}\big)= \min\{ v(r_{i}m_{i}) \ | \ i \leq n \} = v(r_{j}m_{j})}$.
\end{enumerate}
\end{lemma}
\begin{proof}
The first and the second statement are proved in \cite[Lemma 12.16]{HHM}. The third one is a standard computation that we leave to the reader. 
\end{proof}

\subsection{Valued fields and relative quantifier elimination}\label{lingue}
In this section, we summarize many results on valued fields that will be used through the paper. There are many languages in which one can view a valued field $K$ to obtain field quantifier elimination statements, we introduce them in detail and state their corresponding relative quantifier elimination. In this paper we are only concerned about henselian valued fields of equicharacteristic zero.

\begin{definition}\label{1}[The $\mathcal{L}_{\val}$-language]
 Let $(K, \mathcal{O})$ be a henselian valued field of equicharacteristic zero, where $\mathcal{O}$ is the valuation ring and $\mathcal{M}$ is the maximal ideal of $\mathcal{O}$.  Any henselian valued field can be seen as a three sorted structure $(K, \mathbf{k}, \Gamma)$ where the first two sorts are equipped with the language of fields while the third one is provided with the language of ordered abelian groups $\mathcal{L}_{OAG}=\{ 0, <, +, -\}$ extended by a constant $\infty$. We interpret the first sort as the main field, the second one as the residue field and the third one as the value group $\Gamma$ with a constant $\infty$, where $\gamma< \infty$ for all $\gamma \in \Gamma$ and $\gamma+\infty= \infty+\gamma=\infty$ for all $\gamma \in \Gamma \cup \{ \infty \}$. We add a function symbol $v: K \rightarrow \Gamma \cup \{ \infty\}$ by extending the valuation to a monoid morphism sending $v(0)=\infty$. We add as well a function symbol $\res: K \rightarrow \mathbf{k}$, which sends an element $a \in \mathcal{O}$ to its residue class $\res(a)=a +\mathcal{M}$, while $\res(a)=0$ for any element $a \in K \backslash \mathcal{O}$. We refer to this language as the $\mathcal{L}_{\val}$-language.
\end{definition}
\begin{notation} Let $L$ be a henselian valued field we will denote as $k_{L}$ its residue field and $\Gamma_{L}$ its value group. 
\end{notation}
Certain classes of henselian valued fields of equicharacteristic zero eliminate field quantifiers in the $\mathcal{L}_{\val}$-language. For example, the following is a well known fact.
\begin{theorem}\label{QERACF} Let $K$ be a henselian valued field of equicharacteristic zero with residue field algebraically closed, then $K$ eliminates quantifiers relative to the value group in the language $\mathcal{L}_{\val}$. 
\end{theorem}
An immediate consequence of this theorem is the following statement. 
\begin{corollary}\label{seracf} Let $(K,\mathbf{k},\Gamma \cup \{ \infty\},\res,v)$ be a henselian valued field of equicharacteristic zero with residue field algebraically closed, then $\mathbf{k}$ and $\Gamma$ are purely stably embedded and orthogonal to each other.
\end{corollary}

 \begin{definition} Let $(K,\mathbf{k},\Gamma)$ be a valued field an \emph{angular component map} is a map $\ac:K \rightarrow \mathbf{k}$ that satisfies the following conditions:
 \begin{itemize}
 \item $\ac(0)=0$,
 \item for all $x \in \mathcal{O}^{\times} \ $ $\ac(x)=x+\mathcal{M}= \res(x)$, 
 \item for all $x,y \in K \ $ $\ac(xy)=\ac(x) \ac(y)$. 
 \end{itemize}
 \end{definition}
 \begin{definition}\label{2}[The $\mathcal{L}_{\ac}$-language] We denote by $\mathcal{L}_{\ac}$ the expansion of $\mathcal{L}_{\val}$ where an angular component map is added to the language. 
 \end{definition}
In \cite[Theorem 4.1]{Pas} Pas proved that any henselian valued field of equicharacteristic zero eliminates field quantifiers in the $\mathcal{L}_{\ac}$-language, we include the statement for sake of completeness.\\
 Let $\mathcal{K}=(K,\mathbf{k},\Gamma, \res, v, \ac)$ be a valued field of equicharacteristic zero. A good substructure of $\mathcal{K}$ is a triple $\mathcal{E}=(E,\mathbf{k}_{\mathcal{E}}, \Gamma_{\mathcal{E}})$ such that: 
 \begin{itemize}
 \item $E$ is a subfield of $K$,
 \item $\mathbf{k}_{\mathcal{E}}$ is a subfield of $\mathbf{k}$ with $\ac(E) \subseteq \mathbf{k}_{\mathcal{E}}$ (In particular, $\res(\mathcal{O}_{E}) \subseteq \mathbf{k}_{\mathcal{E}}$), 
 \item $\Gamma_{\mathcal{E}}$ is an ordered abelian subgroup of $\Gamma$ with $v(E^{\times}) \subseteq \Gamma_{\mathcal{E}}$. 
 \end{itemize}
 \begin{definition} Let $\mathcal{K}$ and $\mathcal{K}'$ be henselian valued fields of equicharacteristic zero seen as $\mathcal{L}_{\ac}$-structures and let $\mathcal{E}=(E,\mathbf{k}_{\mathcal{E}}, \Gamma_{\mathcal{E}}), \mathcal{E}'=(E',\mathbf{k}_{\mathcal{E}'}, \Gamma_{\mathcal{E}'})$ be good substructures of $\mathcal{K}$ and $\mathcal{K}^{'}$ respectively. A triple $\mathbf{f}=(f, f_{r}, f_{v})$ is said to be a good map, if $f: E \rightarrow E'$ and  $f_{r}: \mathbf{k}_{\mathcal{E}} \rightarrow \mathbf{k}_{\mathcal{E}'}$ are field isomorphisms and $f_{v}:\Gamma_{\mathcal{E}} \rightarrow \Gamma_{\mathcal{E}'}$ is a $\mathcal{L}_{OAG}$- ordered group isomorphism such that:
 \begin{itemize}
 \item $f_{r}(\ac(a))=\ac'(f(a))$ for all $a \in E$ and $f_{r}$ is a partial elementary map between the fields $\mathbf{k}$ and $\mathbf{k}'$,
 \item $f_{v}(v(a))= v^{'}(f(a))$ for all $a \in E^{\times}$, and $f_{v}$ is a partial elementary map between the ordered abelian groups $\Gamma$ and $\Gamma'$. 
 \end{itemize}
 \end{definition}
 \begin{theorem}\label{Pas}[Pas] Let $\mathcal{K}$ and $\mathcal{K}'$ be two henselian valued fields of equicharacteristic zero in the $\mathcal{L}_{\ac}$-language. Let $\mathbf{f}: \mathcal{E} \rightarrow \mathcal{E}'$ be a good map, then $\mathbf{f}$ is elementary. 
 \end{theorem}
The following statement is an immediate consequence of Theorem \ref{Pas}.
\begin{corollary}\label{seac} Let $\mathcal{K}=(K,\mathbf{k}, \Gamma, \res, v, \ac)$ be a henselian valued field of equicharacteristic zero. Then the residue field and the value group are purely stably embedded and are orthogonal to each other. 
\end{corollary}
 Given $(K, \mathbf{k},\Gamma)$ a valued field we denote as $RV^{\times}$ the multiplicative quotient group $K^{\times}/ (1+\mathcal{M})$ and $\rv:K^{\times} \rightarrow RV^{\times}$ the canonical projection map. By adding a constant $0_{RV}$ we can naturally extend this map sending the element $0$ to $0_{RV}$, so we denote $RV= RV^{\times} \cup\{ 0_{RV}\}$.  For any $a \in \mathcal{O} \backslash \mathcal{M}$ the class $a(1 +\mathcal{M})$ only depends on the coset $a+\mathcal{M}$, so we obtain a group embedding $i:\mathbf{k}^{\times} \rightarrow RV^{\times}$ by sending the element $a +\mathcal{M} \in \mathbf{k}^{\times}$ to $a(1+\mathcal{M}) \in RV^{\times}$. We can also consider the group morphism $v_{\rv}: RV^{\times} \rightarrow \Gamma$ induced by the valuation map $v: K^{\times} \rightarrow \Gamma$, defined as $v_{\rv}(a(1+\mathcal{M}))= v(a)$. In fact, given two elements in the main field sort $a,b \in K$ if  $a(1+\mathcal{M})=b(1+\mathcal{M})$ then $v(a)=v(b)$. Therefore, we have a pure exact sequence:
 \begin{equation*}
 1 \rightarrow \mathbf{k}^{\times} \rightarrow RV^{\times} \rightarrow \Gamma \rightarrow 0,
 \end{equation*}
 which can be naturally extended to a short exact sequence of monoids by adding some constants, i.e. 
 \begin{equation*}
 1 \rightarrow \mathbf{k} \rightarrow RV \rightarrow \Gamma \cup \{\infty\} \rightarrow 0.
 \end{equation*}
 Besides the induced multiplication, $RV$ also inherits a partially defined addition from $K$, via the ternary relation:
 \begin{align*}
 \oplus (a,b,c) \leftrightarrow \exists x,y,z \in K (a=\rv(x) \wedge b=\rv(y) \wedge c=\rv(z) \wedge x+y=z). 
 \end{align*}
We consider the three sorted structure $(\mathbf{k}, RV, \Gamma \cup\{\infty\})$ with the language $\mathcal{L}_{rv}= \mathcal{L}_{r} \cup \mathcal{L}_{g} \cup \{ \cdot_{rv}, i, v_{rv}\}\label{rvok}$, where $\mathcal{L}_{r}$ is a copy of the language of fields for the first sort, $\mathcal{L}_{g}$ is the language of ordered abelian groups extended by a constant $\infty$ i.e. $\{ 0_{g}, +_{g}, -_{g}, <_{g}, \infty\}$, $i$ is a function symbol interpreted as the monoid morphism $i: \mathbf{k} \rightarrow RV$ and $v_{\rv}$ is a function symbol interpreted as the monoid  morphism $v_{\rv}: RV \rightarrow \Gamma$. \\
Building on work of Basarab in \cite{Basarab}, Kuhlmann proved in \cite{Kuhlmann} that any henselian valued field of equicharacteristic zero eliminates field quantifiers relative to the structure $(\mathbf{k},RV, \Gamma)$. We use Flenner as a reference, the following is \cite[Proposition 3.3.1]{Flenner}.
\begin{proposition}\label{relativeRV} [Kulhmann] Let $K$ be a henselian valued field of equicharacteristic zero, the theory of $(K, \mathbf{k}, RV, \Gamma \cup \{\infty\}, \res, v, v_{\rv}, i)$ eliminates field quantifiers. 
\end{proposition}
The following is a direct consequence of the relative quantifier elimination to $RV$. 
\begin{corollary}\label{seRV}Let $K$ be a henselian valued field of equicharacteristic zero, the structure $(\mathbf{k}, RV, \Gamma \cup \{ \infty\})$ is purely stably embedded. 
\end{corollary}
Kulhmann's statement reduces the elimination of field quantifiers to the structure $(\mathbf{k},RV, \Gamma \cup \{\infty\})$. For certain classes of henselian valued fields of equicharacteristic zero the structure $(\mathbf{k}, RV, \Gamma \cup \{ \infty\})$ eliminates $RV$ quantifiers in the language $\mathcal{L}_{\rv}$. 
\begin{proposition}\label{RVRACF} Let $K$ be a henselian valued field of equicharacteristic zero with residue field algebraically closed, the structure $(\mathbf{k}, RV, \Gamma \cup \{\infty\})$ eliminates quantifiers relative to the value group in the language $\mathcal{L}_{\rv}$.
\end{proposition}
\begin{proof}
This follows by a standard back and forth argument using that $k^{\times}$ is divisible. 
\end{proof}

The elimination of $RV$ quantifiers in the more general setting was later obtained by Aschenbrenner, Chernikov, Gehret and Ziegler in \cite{distal}.  They extend the language adding a new sort for each $n \in \mathbb{N}$ denoted as $\mathbf{k}^{\times}/(\mathbf{k}^{\times})^{n}$ which is an abelian group and we extend it adding an element $\infty$ such that for each $a \in \mathbf{k}^{\times}/(\mathbf{k}^{\times})^{n}$, $a \cdot \infty=\infty$. For each $n \in \mathbb{N}$ we denote this sort as $\mathcal{A}_{n}$, and we refer to the multi-sorted structure $\mathcal{A}=(\mathcal{A}_{n} \ | \ n \in \mathbb{N})$ as the \emph{power residue sorts}. We add surjective maps $\pi_{n}: \mathbf{k}^{\times} \rightarrow \mathbf{k}^{\times}/ (\mathbf{k}^{\times})^{n}$, which can be naturally extended to a monoid morphism $\pi_{n}: \mathbf{k} \rightarrow \mathcal{A}_{n}$. We add maps $\rho_{n}: RV \rightarrow \mathcal{A}_{n}$, interpreted as $\rho_{n}(0)= \infty$, over $v_{rv}^{-1}(n\Gamma)$, we define $\rho_{n}$ as the composition of the group morphisms:
\begin{align*}
 v_{rv}^{-1}(n\Gamma) \subseteq RV  \rightarrow RV^{n} \cdot i(\mathbf{k}^{\times}) \rightarrow (RV^{n} \cdot i(\mathbf{k}^{\times}))/ RV^{n} \cong \mathbf{k}^{\times}/(RV^{n} \cap \mathbf{k}^{\times}) \cong \mathbf{k}^{\times}/ (\mathbf{k}^{\times})^{n}, 
\end{align*}
and the map is equal to zero outside of $v_{rv}^{-1}(n\Gamma)$. Let $\displaystyle{\mathcal{L}_{rvqe}=\mathcal{L}_{rv} \cup \{ \rho_{n}, \pi_{n} \ | \ n \in \mathbb{N}\}}$. \\
The following is \cite[Remark 4.4]{distal}.
\begin{corollary} \label{RVQE}The structure $(\mathcal{A}, RV, \Gamma \cup \{ \infty\}, \{\pi_{n} ,\rho_{n} \ | \ n \in \mathbb{N}\})$ eliminates $RV$ quantifiers. In particular, $\mathcal{A}$ and $\Gamma \cup \{\infty \}$ are purely stably embedded and are orthogonal to each other. 
\end{corollary}
Combining Proposition \ref{relativeRV} and Corollary \ref{RVQE} Aschenbrenner, Chernikov, Gehret and Ziegler obtained as well a field quantifier elimination for henselian valued fields of equicharacteristic zero relative to the power residue sorts and the value group in the following language:
\begin{definition}\label{3}[The Language $\mathcal{L}$] Consider the expansion of $\mathcal{L}_{\val}$ obtained by adding the power residue sorts $\mathcal{A}=(\mathcal{A}_{n} \ | \ n \in \mathbb{N})$. We also add the surjective maps $\pi_{n}:\mathbf{k} \rightarrow \mathcal{A}_{n}$, and we interpret them as the described. For each $n \in \mathbb{N}$ we add a map $\res^{n}: K \rightarrow \mathcal{A}_{n}$ interpreted in the following way: if $v(a) \notin n\Gamma$ set $\res^{n}(a)=0$. Otherwise, let $b$ be any element of $K$ with $nv(b)=v(a)$ and set $\res^{n}(a)= \pi_{n} \big( \res\big( \frac{a}{b^{n}}\big)\big)$. We denote this expansion of $\mathcal{L}_{\val}$ by $\mathcal{L}$. 
\end{definition}
Note that for each $a \in K$, $\res^{n}(a)=\rho_{n}(rv(a))$. The following is a direct consequence of \cite[Theorem 5.15]{distal}.
\begin{theorem} \label{QE} A henselian valued field of equicharacteristic zero eliminates field quantifiers in the language $\mathcal{L}$. 
\end{theorem}
The following statement is an immediate consequence of \cite[Theorem 5.15]{distal}.
\begin{corollary}\label{stL} Let $K$ be a henselian valued field of equicharacteristic zero seen as a $\mathcal{L}$-structure. The power residue sorts and the value group are purely stably embedded and are orthogonal to each other. 
\end{corollary}

\begin{definition}[The $\mathcal{L}_{RV}$-language] We write $\mathcal{L}_{RV}$ to denote the extension of the language $\mathcal{L}$ where we add as well a sort for the monoid $RV$ where we equip the exact sequence $(\mathbf{k}, RV, \Gamma \cup \{ \infty\})$ with the language $\mathcal{L}_{rvqe}$. 
\end{definition}
\subsubsection{Some remarks in ordered abelian groups}
In $1984$ Gurevich and Schmitt \cite{NIPOAG} showed that every ordered abelian group is $NIP$. In \cite{schmitt}, Schmitt investigated deeply the model completeness of theories of ordered abelian groups and obtained a (\emph{relative}) quantifier elimination to the spines, whose description can be found in  \cite[Section 2]{NIPOAG}.  Later, Cluckers and Halupczok in \cite[Definition 1.5]{CluckersImmi} introduced a language $\mathcal{L}_{CH}$-extending $\mathcal{L}_{OAG}=\{ +,-, 0, <\}$ and obtained a $(relative)$ quantifier elimination to the \emph{auxiliary sorts}, whose definition can be found in \cite[Section 1.2]{CluckersImmi}. The language $\mathcal{L}_{CH}$ has been more often used by the model theory community as it is more in line with Shelah's imaginary expansion. The following is \cite[Corollary 1.10]{CluckersImmi}.\\
\begin{corollary}\label{linear}Let $G$ be an ordered abelian group, for any function $f: G^{n} \rightarrow G$ which is $\mathcal{L}_{OAG}$-definable with parameters from a set $B$, there exists a partition of $G^{n}$ into finitely many $B$-definable sets and for each such set $A$, $f$ is linear. This is, there are finitely many elements $r_{1},\dots, r_{n},s \in \mathbb{Z}$ with $s \neq 0$ and $b \in \dcl(B)$ such that for any $\mathcal{a} \in A$ we have $\displaystyle{f(a_{1},\dots, a_{n})= \frac{1}{s} \big(\sum_{i \leq n} r_{i} a_{i}+b \big)}$. 
\end{corollary}
Let $G$ be an ordered abelian group, we extend the language $\mathcal{L}_{CH}$ by adding a set of constants $\mathcal{C}$ to name each element of $\dcl(\emptyset)$, and we denote this extension as $\mathcal{L}_{CH}^{*}$.  An immediate consequence is the following fact. 
\begin{fact}\label{dcllinear} Let $G$ be an ordered abelian group and $B \subseteq G$. Then $\dcl(B)=\big( \mathbb{Q} \otimes B \big) \cap G$. 
\end{fact}
This fact will play a fundamental role to weaken the necessary hypothesis to obtain domination results for henselian valued fields of equicharacteristic zero. We will denote as $\mathcal{L}_{\val}^{*}$ and $\mathcal{L}^{*}_{\ac}$ the extension of the language obtain by adding a set of constants to the main field  $\Sigma=\{ t_{d} \ | \ d \in \dcl(\emptyset) \cap \Gamma\}$ such that $v(t_{d})=d$ for each element $d \in \dcl(\emptyset) \cap \Gamma$. \\
The following is \cite[Proposition 5.1]{dpminimal}.
\begin{proposition} An ordered abelian group is dp-minimal if and only if for every prime number $p$, $[\Gamma: p\Gamma]$ is finite. 
\end{proposition}
\subsection{Independence notions}
 Trough this paper we will use several notions of independence. We begin by recalling a few basic properties of forking. 
\begin{definition} A formula $\phi(x,b)$ \emph{divides} over $C$ if there is a sequence $(b_{i})_{i<\omega}$ in $\tp(b/C)$ with $b=b_{0}$ such that $\{ \phi(x,b_{i}) \ | \ i <\omega\}$ is $m$-inconsistent. We say that $\phi(x,b)$ \emph{forks} if $\displaystyle{\phi(x,b) \vdash \bigvee_{i \leq k} \psi_{i}(x,b_{i})}$ where each formula $\psi_{i}(x,b_{i})$ divides over $C$. We say that $\tp(a/Cb)$ forks (respectively divides) over $C$ if some formula in the type forks (or divides) over $C$. We write  $a \ind_{C}b$ if the type $\tp(a/Cb)$ does not fork over $C$, and write $a \ind_{C}^{d}b$ to indicate that the type $\tp(a/Cb)$ divides over $C$.
\end{definition}
In many theories the relation of forking independence have been completely characterized. For example, in the theory of algebraically closed fields, forking independence coincides with algebraic independence. Let $C,E$ and $F$ be fields, and suppose that $C \subseteq E \cap F$ we will write $E \ind_{C}^{alg} F$ to indicate that $E$ and $F$ are algebraically independent over $C$. 
\subsubsection{Forking independence in abelian groups}
In \cite{Mike}, the model theory of modules is extensively studied. We will be interested in applying some of the results in \cite{Mike} to the reduct of the value group to the language of groups $\mathcal{L}_{AG}=\{ +, -, 0\}$. It is well known that modules are stable, and every abelian group is a $\mathbb{Z}$-module. We recall some of the necessary notions to characterize forking independence in abelian groups. Throughout this section we consider the $\mathcal{L}_{AG}$ first order theory of some torsion free group and we denote as $\mathfrak{G}$ its monster model.\\
We recall some of the well known facts about stable theories. 
 \begin{fact} \label{nonforkingheir} Let $T$ be a complete first order theory and assume that $T$ is stable and $M \vDash T$. Let $p \in S_{n}(M)$ then $p$ is stationary. Furthermore, for any set of parameters $M \subseteq A$ and $q \in S_{n}(A)$ such that $p \subseteq q$ the following conditions are equivalent:
 \begin{enumerate}
\item  $q$ is a non-forking extension of $p$,
\item $q$ is a heir extension of $p$ (i.e. every formula represented in $q$ is also represented in $p$),
\item  $q$ is a co-heir extension of $p$ (i.e. for every formula $\phi(\mathbf{x},\mathbf{a})\in q$ is finitely satisfiable in $M$, this is there is some $\mathbf{m} \subseteq M$ such that $\vDash \phi(\mathbf{m}, \mathbf{a})$.)
\end{enumerate}
 \end{fact}
\begin{definition}[\emph{p.p. formula}] A $p.p.$ formula $\phi(\mathbf{v})$ is an $\mathcal{L}_{AG}$ formula of the form 
\begin{align*}
\exists w_{1},\dots, w_{l} \big( \bigwedge_{j=1}^{k} \sum_{i=1}^{n} r_{j_{i}} v_{i}+ \sum_{i=1}^{l}s_{j_{i}} w_{i}=0\big),
\end{align*}
 where $s_{j_{i}}$, $r_{j_{i}} \in \mathbb{Z}$, and $\mathbf{v}= (v_{1},\dots,v_{n})$ is a tuple of variables. 
\end{definition}
 Given a $p.p.$ formula if we replace the last $(n-i)$- variables  by a tuple of parameters $\bar{a}=(a_{i},\dots,a_{n})$, the formula $\phi(v_{1},\dots,v_{i-1}, \bar{a})$ defines a coset of $\phi(v_{1},\dots,v_{i-1}, \bar{0})$, which defines a subgroup of $\mathfrak{G}^{i}$. 
 \begin{definition}[$p.p.$-type] Let $\mathbf{c}$ be a tuple and $A$ some set of a parameters, the $p.p.$ type of $\mathbf{c}$ over $A$ is the set of $p.p.$ $\mathcal{L}_{AG}(A)$-formulas that $\mathbf{c}$ satisfies. This is:
 \begin{align*}
 \pp(\mathbf{c}/A)=\{\phi(\mathbf{v}, \mathbf{a})\ | \ \text{ $\phi(\mathbf{v}, \mathbf{a})$ is a $\mathcal{L}_{AG}(A)$ p.p. formula and $\vDash \phi(\mathbf{c}, \mathbf{a})$} \}.
 \end{align*}
 \end{definition}
 If $p$ is a $p.p.$-type over $A$ and $\phi(\mathbf{v}, \mathbf{y})$ is an $\mathcal{L}_{AG}$-formula, then we say that it is \emph{represented} in $p$ if there is some tuple $\mathbf{a} \subseteq A$ such that $\phi(\mathbf{v}, \mathbf{a}) \in p$. We consider the type definable group 
 \begin{equation*}
 G(p)=\{ \phi(\mathbf{v}, \bar{0}) \ | \ \phi(\mathbf{v}, \mathbf{y})\ \text{is represented in $p\}$}.
 \end{equation*}
It is well known that in stable theories to characterize the non-forking extensions of $p$ the group $G(p)$ would not be the right invariant to consider, but instead one might be more interested in its connected component $\displaystyle{G^{0}(p)= \bigcap_{H \in F} H}$,  where 
\begin{equation*}
\mathcal{F}=\{ H \ | \ H \ \text{is a subgroup of some} \ G \in G(p) \ \text{and} \ [G: H] \ \text{is finite} \}.
\end{equation*}
The following is \cite[Theorem 5.3]{Mike}.
 \begin{theorem} Let $p$ be a type and suppose that $q$ is any extension of $p$. Then $q$ is a non-forking extension of $p$ if and only if $G^{0}(p)=G^{0}(q)$. In particular, for any type $p$ if $G(p)=G^{0}(p)$ then $p$ is stationary. 
 \end{theorem}
This statement allows us to characterize forking independence for arbitrary set of parameters.
 \begin{corollary}Let $A,B,C \subseteq \mathfrak{G}$, then $B \ind_{A} C$ if and only if for every $p.p$-formula $\psi(\mathbf{x}, \mathbf{y}, \mathbf{z})$ and tuples $\mathbf{b} \subseteq B, \mathbf{c} \subseteq C$ and $\mathbf{a} \subseteq A$ such that $\vDash \psi(\mathbf{b},\mathbf{c},\mathbf{a})$, there is some $p.p.$ formula $\phi(\mathbf{x},\mathbf{w})$ and a tuple $\mathbf{a}' \subseteq A$ such that $\vDash \phi(\mathbf{b},\mathbf{a}')$ and $[\phi(\mathbf{x},\bar{0}): \phi(\mathbf{x}, \bar{0}) \wedge \psi(\mathbf{x}, \bar{0},\bar{0})] \ \text{is finite.}$. 
 \end{corollary}
 
 We conclude this subsection by introducing the notion of independence that we will be using for the value group. Let $(\Gamma,+,-,\leq, 0)$ be a non-trivial ordered abelian group. Let $T$ be its complete $\mathcal{L}_{OAG}$-theory and $\mathfrak{G}$ be its monster model. We let $\mathfrak{G}\upharpoonright_{\mathcal{L}_{AG}}$ be its reduct to the language of abelian group, this is purely a torsion free abelian group. 
 \begin{definition}  Let $A, B, C \subseteq \mathfrak{G}$, then $A \ind_{C}^{s} B$ if and only if $tp\upharpoonright_{\mathcal{L}_{AG}}(A/BC)$ does not fork over $C$ if and only if $A \ind_{C} B$ in the stable structure $\mathfrak{G}\upharpoonright_{\mathcal{L}_{AG}}$. 
  \end{definition}
 We will use the following fact repeatedly.
 \begin{fact}\label{intersection} Let $A$ and $B$ be subgroup of $\mathfrak{G}$ and let $C\subseteq A \cap B$ be a common subgroup. If $A \ind_{C}^{s} B$, then $A \cap B \subseteq \dcl(C)$. 
\end{fact}
\begin{proof}
This follows by forking independence for stable formulas, including $x=y$. If $a \in A \cap B \backslash \acl(C)$, then the formula $x=a \in \tp(A/B)$ and divides over $C$, because we can find an infinite non constant indiscernible sequence $(a_{i})_{i=0}^{\infty}$ in $\tp(a/C)$ and $\{ x=a_{i}\}_{i=0}^{\infty}$ is $2$-inconsistent.  We conclude that $a$ must be algebraic over $C$, and as there is a definable order in the home sort $a \in \dcl(C)$.
%Let $a \in A$ and $b \in B$ be such that $a=b$, then the formula $\phi(x,y):= (x-y)=0$ is a $p.p.$ formula such that $\vDash \phi(a,b)$, and $\phi(x,0)$ defines the trivial subgroup $\{0\}$ of $\mathfrak{G}$. As $\tp\upharpoonright_{\mathcal{L}_{AG}}(a/B)$ is a non-forking extension of $\tp\upharpoonright_{\mathcal{L}_{AG}}(a/C)$ there is some $p.p.$ formula $\psi(x,z)$ and some set of parameters $\bar{c}$ in $C$ such that $\vDash \psi(a, \bar{c})$ and $[\psi(x,0): \psi(x,0) \wedge \phi(x,0)]$ is finite. Thus $\psi(x,0)$ defines a finite subgroup of $\mathfrak{G}$, and as a result $a$ is $\mathcal{L}_{AG}$-algebraic over $C$. As $\mathfrak{G}$ is linearly ordered, $\acl(C)=\dcl(C)$ so $a \in  \dcl(C)$, as required. 
\end{proof}

\section{Domination by the power residue sorts and the value group}\label{dominationvalres}
In this section we study domination of the type of a valued fields by the power residue sorts and the value group in each of the languages. We would like to highlight that the required ingredients to carry out the argument are the existence of separated basis and a relative quantifier elimination statement. We obtain the following results:
\begin{enumerate}
\item In the $\mathcal{L}$-language the type of a valued field over a maximal model is dominated by the power residue sorts and the value group.
\item In the $\mathcal{L}_{\ac}$-language the type of a valued field over a maximal field is dominated by the residue sort and the value group.
\item For the theory of henselian valued fields of equicharacteristic zero with residue field algebraically closed, the type of a valued field over a maximal field is dominated by the residue sort and the value group in the $\mathcal{L}_{\val}^{*}$-language.
\end{enumerate}

The following is \cite[Lemma 2.5]{Haskell}.% \textcolor{red}{Can I find a better reference?}
\begin{proposition}\label{linearlydisjoint} Let $L$ and $M$ be valued fields with $C \subseteq L \cap M$ be a common valued subfield. Assume that:
\begin{enumerate}
    \item $\Gamma_{L} \cap \Gamma_{M}=\Gamma_{C}$, 
    \item $k_{M}$ and $k_{L}$ are linearly disjoint over $k_{C}$,
    \item $M$ (or $L$) have the good separated basis property over $C$.
\end{enumerate}
Then $M$  (or $L$) has the separated basis property over $L$ (or $M$ respectively). Therefore, $L$ and $M$ are linearly disjoint over $C$, $\Gamma_{C(L,M)}$ is the group generated by $\Gamma_{L}$ and $\Gamma_{M}$ over $\Gamma_{C}$ and $k_{C(L,M)}$ is the field generated by $k_{M}$ and $k_{L}$ over $k_{C}$. 
\end{proposition}
The following is a direct consequence of Proposition \ref{linearlydisjoint}, but we bring details to the picture to clarify some of our arguments. 
\begin{fact}\label{residueok} Let $L$ and $M$ be valued fields with $C \subseteq L \cap M$ a common valued subfield. Assume that:
\begin{enumerate}
    \item $\Gamma_{L} \cap \Gamma_{M}=\Gamma_{C}$, 
    \item $k_{M}$ and $k_{L}$ are linearly disjoint over $k_{C}$,
    \item $M$ (or $L$) has the good separated basis property over $C$.
\end{enumerate}
Let $a \in \mathcal{O}_{C(L,M)}^{\times}$ then there are elements $l^{1}_{1}, \dots, l^{1}_{k}, l^{2}_{1}, \dots, l^{2}_{s}, l \in \mathcal{O}_{L}$ and  $m^{1}_{1}, \dots, m^{1}_{k}, m^{2}_{1}, \dots, m^{2}_{s}, m \in \mathcal{O}_{M}$, such that:
\begin{align*}
\res(a)= \big( 1 + \sum _{i \leq k} \res(l_{i}^{1}) \res(m_{i}^{1}) \big) \big( 1 + \sum _{i \leq s} \res(l_{i}^{2}) \res(m_{i}^{2}) \big)^{-1} \res(l) \res(m).  
\end{align*}
\end{fact}
\begin{proof} Let $a= \frac{y_{1}}{y_{2}}$ where $y_{1}, y_{2} \in C[L,M]$. By Proposition \ref{linearlydisjoint} $L$ and $M$ are linearly disjoint over $C$ and $M$ (or $L$) has the separated basis property over $L$ (or $M$). Suppose that $\displaystyle{y_{i}=\sum_{j \leq n_{i}}\hat{l}_{j}^{i}\hat{m}_{j}^{i}}$, by hypothesis $v(y_{1})=v(y_{2})=\gamma$. As $M$ (or $L$) has the separated basis property over $L$ (or $M$) there is some index $j^{i}_{0}$ such that $\gamma= v(\hat{l}_{j^{i}_{0}}^{i} \hat{m}_{j^{i}_{0}}^{i})$ and let $I_{i}=\{ j \leq n_{i}\ | \ v(\hat{l}_{j}^{i} \hat{m}_{j}^{i})=\gamma\}$. Hence  
\begin{equation*}
y_{i}= \hat{l}_{j^{i}_{0}}^{i} \hat{m}_{j^{i}_{0}}^{i} \big( 1 + \sum_{j \leq n_{i}, j \neq j^{i}_{0}} \big(\frac{\hat{l}_{j^{i}}^{i}\hat{m}_{j^{i}}^{i}}{\hat{l}_{j_{0}^{i}}^{i}\hat{m}^{i}_{j_{0}^{i}}}\big)\big).
\end{equation*}
\begin{claim}\label{chumbi} Given elements $l \in L$ and $m \in M$ such that $v(lm)=0$ there is some element $c \in C$ satisfying $v(lc)=0$ and $v(mc^{-1})=0$.\end{claim}
\begin{proof}
Let $l \in L$ and $m \in M$ be such that $v(lm)=0$ then $v(l)=-v(m) \in \Gamma_{L} \cap \Gamma_{M}=\Gamma_{C}$ so there is some $c \in C$ such that $v(lc)=0$ and $v(mc^{-1})=0$.
\end{proof}

In particular, for any $j \in I_{i}$, since $v(\hat{l}_{j}^{i} \hat{m}_{j}^{i})= v(\hat{l}_{j_{0}^{i}}^{i} \hat{m}_{j_{0}^{i}}^{i})$ we have that $v\big( \frac{\hat{l}_{j}^{i}}{\hat{l}_{j_{0}^{i}}^{i}} \frac{\hat{m}_{j_{0}^{i}}^{i}}{\hat{m}_{j}^{i}}\big)=0$, so we can find elements $c_{j}^{i} \in C$ such that  $v\big( \frac{\hat{l}_{j}^{i}}{\hat{l}_{j_{0}^{i}}^{i} } c_{j}^{i} \big)=0$ and $ v \big(  (c_{j}^{i})^{-1}\frac{\hat{m}_{j_{0}^{i}}^{i}}{\hat{m}_{j}^{i}}\big)=0$. Let $l_{j}^{i}= \frac{\hat{l}_{j}^{i}}{\hat{l}_{j_{0}^{i}}^{i} }c_{j}^{i}$ and $m_{j}^{i}=(c_{j}^{i})^{-1}\frac{\hat{m}_{j_{0}^{i}}^{i}}{\hat{m}_{j}^{i}}$. Moreover,
 $\gamma= v(l_{j_{0}^{1}}^{1}m_{j_{0}^{1}}^{1})= v(l_{j_{0}^{2}}^{2}m_{j_{0}^{2}}^{2})$, thus we can find an element $c \in C$ 
 such that $v\big( \frac{l_{j_{0}^{1}}^{1}}{l_{j_{0}^{2}}^{2}}c \big)=0$ and $v\big( \frac{\hat{m}_{j_{0}^{1}}^{1}}{\hat{m}_{j_{0}^{2}}^{2}}c^{-1} \big)=0$, we set $l=\frac{\hat{l}_{j_{0}^{1}}^{1}}{\hat{l}_{j_{0}^{2}}^{1}}c$ and $m=\frac{\hat{m}_{j_{0}^{1}}^{1}}{\hat{m}_{j_{0}^{2}}^{2}}c^{-1}$.\\
Then: 
\begin{align*}
\res(a)&=\res(\frac{y_{1}}{y_{2}})=\res \big( 1 + \sum_{j\leq n_{1}, j\neq j_{0}^{1}} \frac{\hat{l}_{j}^{1} \hat{m}_{j}^{1}}{\hat{l}_{j_{0}^{1}}^{1}\hat{m}_{j_{0}^{1}}^{1}} \big) \big(\res \big( 1 + \sum_{j\leq n_{2}, j\neq j_{0}^{2}} \frac{\hat{l}_{j}^{2} \hat{m}_{j}^{2}}{\hat{l}_{j_{0}^{2}}^{2}\hat{m}_{j_{0}^{2}}^{2}} \big)\big)^{-1} \res \big( \frac{\hat{l}_{j_{0}^{1}}^{1} \hat{m}_{j_{0}^{1}}^{1}}{\hat{l}_{j_{0}^{2}} \hat{m}_{j_{0}^{2}}^{2}}\big)\\
&= \res \big(1+\sum_{j \in I_{1}, j \neq j_{0}^{1}} \big(\frac{\hat{l}_{j}^{1}\hat{m}_{j}^{1}}{\hat{l}^{1}_{j_{0}^{1}} \hat{m}^{1}_{j_{0}^{1}}} \big) \big) \big(\res \big(1+\sum_{j \in I_{2}, j \neq j_{0}^{2}} \big(\frac{\hat{l}_{j}^{2}\hat{m}_{j}^{2}}{\hat{l}^{2}_{j_{0}^{2}} \hat{m}^{2}_{j_{0}^{2}}}\big)\big)\big)^{-1} \res(l) \res(m)\\
&=\big(1+\sum_{j \in I_{1}, j \neq j_{0}^{1}} \res \big(\frac{\hat{l}_{j^{1}}^{1} \hat{m}_{j^{1}}^{1}}{\hat{l}_{j_{0}^{1}}^{1} \hat{m}_{j_{0}^{1}}^{1}}\big)\big) \big(1+\sum_{j \in I_{2}, j \neq j_{0}^{2}} \res \big(\frac{\hat{l}_{j^{2}}^{2} \hat{m}_{j^{2}}^{2}}{\hat{l}_{j_{0}^{2}}^{2} \hat{m}_{j_{0}^{2}}^{2}}\big)\big)^{-1} \res(l)\res(m)\\
&= \big(1+ \sum_{j \in I_{1}, j \neq j_{0}^{1}} \res(l_{j}^{1})\res(m_{j}^{1})\big) \big(1+\sum_{j \in I_{2}, j \neq j_{0}^{2}} \res(l_{j}^{2}) \res(m_{j}^{2})\big)^{-1} \res(l)\res(m) \ \text{as required.}
\end{align*}
\end{proof}
We start by recalling some propositions about extending a given isomorphism of valued fields in the reduct $\mathcal{L}_{\val}$. The following is \cite[Proposition 2.6]{Haskell}. %\textcolor{red}{puedo encontrar una mejor referencia?}
\begin{proposition}\label{isoextended} Let $L$ and $M$ be valued fields with $C \subseteq L \cap M$ a common valued subfield. Assume that $\Gamma_{L} \cap \Gamma_{M}=\Gamma_{C}$, $k_{L}$ and $k_{M}$ are linearly disjoint over $k_{C}$ and that $L$ or $M$ has the good separated basis property. Let $\sigma: L \rightarrow L'$ be a $\mathcal{L}_{\val}$ valued field isomorphism which is the identity on $C\Gamma_{L}k_{L}$. Then $\sigma$ can be extended to a $\mathcal{L}_{\val}$- valued field isomorphism $f$ from $C(L,M)$ to $C(L',M)$ which is the identity on $M$ and $f \upharpoonright_{L}=\sigma$. 
\end{proposition}
We continue arguing that without loss of generality we may assume the $\mathcal{L}_{\val}$-isomorphism to fix the residue field and the value group of $M$ instead. 
\begin{proposition}\label{isoextended2}Let $L$ and $M$ be valued fields with $C \subseteq L \cap M$ a common valued subfield. Assume that:
\begin{enumerate}
    \item $\Gamma_{L} \cap \Gamma_{M}= \Gamma_{C}$, 
    \item $k_{M}$ and $k_{L}$ are linearly disjoint over $k_{C}$,
    \item $M$ or $L$ have the good separated basis property over $C$.
\end{enumerate}
And let $\sigma: C(L,M) \rightarrow C(L',M')$ be an $\mathcal{L}_{\val}$-isomorphism fixing $Ck_{M}\Gamma_{M}$, such that $\sigma(L)=L^{'}$ and $\sigma(M)=M'$. 
 Then there is an $\mathcal{L}_{\val}$-isomorphism $\tau:C(L,M) \rightarrow C(L',M)$ such that $\tau$ is the identity on $M$ and $\tau\upharpoonright _{L}=\sigma \upharpoonright_{L}$. 
 \end{proposition}
\begin{proof}
Let $\sigma:C(L,M) \rightarrow C(L', M')$ be the given $\mathcal{L}_{\val}$ isomorphism fixing $Ck_{M}\Gamma_{M}$. We want to find an $\mathcal{L}_{\val}$-isomorphism  $\tau:C(L, M) \rightarrow C(L', M)$ which is the identity on $M$ and such that $\tau\upharpoonright _{L}=\sigma \upharpoonright_{L}$. We consider the restriction map $\sigma^{-1}\upharpoonright _{M'}: M' \rightarrow M$, an $\mathcal{L}_{\val}$-isomorphism fixing $Ck_{M}\Gamma_{M}$. By Proposition \ref{isoextended} there is an $\mathcal{L}_{\val}$-isomorphism $\phi: C(M', L') \rightarrow C(M, L')$ that extends $\sigma^{-1}\upharpoonright _{M'}: M' \rightarrow M$ and is the identity on $L'$. Let $\tau: C(L,M) \rightarrow C(L',M)$ be the $\mathcal{L}_{\val}$-isomorphism given by the composition $\tau= \phi \circ \sigma$; this map satisfies the required conditions.
\end{proof}
We conclude this section by restating our result in terms of domination for the class of henselian valued fields of equicharacteristic zero with residue field algebraically closed. We first recall a general fact about regular extensions:
\begin{fact}\label{mago} Let $F$ be a field, $E$ a regular field extension of $F$ and $R$ be any other field extension of $F$. If $E \ind^{alg}_{F} R$ then $E$ and $R$ are linearly disjoint over $F$.
\end{fact}
\begin{proof}
See \cite[Theorem $4.12$ Chapter VIII]{Lang}. 
\end{proof}

\begin{corollary}\label{racf}Let $T$ be some $\mathcal{L}_{\val}^{*}$ complete extension of the first order theory of henselian valued fields of equicharacteristic zero with residue field algebraically closed. Let $\mathfrak{M}$ be its monster model and $C$ a maximal field. Let $C \subseteq L$ be a valued field extension such that $\Gamma_{L}/\Gamma_{C}$ is a torsion free extension and $k_{L}$ is a regular extension of $k_{C}$. Then $\tp(L/C)$ is dominated by the residue field and the value group, this is for any field extension $C \subseteq M$ such that $k_{M} \ind^{alg}_{k_{C}} k_{L}$ and $\Gamma_{M} \ind_{\Gamma_{C}}^{s}\Gamma_{L}$ we have $\tp(L/Ck_{M}\Gamma_{M}) \vdash \tp(L/M)$. 
\end{corollary}
\begin{proof} Let $C \subseteq M$ be a field extension such that $k_{M} \ind^{alg}_{k_{C}} k_{L}$ and $\Gamma_{M} \ind_{\Gamma_{C}}^{s}\Gamma_{L}$ . By Fact \ref{mago} $k_{M}$ and $k_{L}$ are linearly disjoint over $k_{C}$. As $\Gamma_{M} \ind_{\Gamma_{C}}^{s} \Gamma_{L}$, by Fact \ref{intersection} $\Gamma_{M} \cap \Gamma_{L} \subseteq \dcl(\Gamma_{C})$. Combining Fact \ref{dcllinear} together with the hypothesis of $\Gamma_{L}/\Gamma_{C}$ is torsion free we obtain that $\Gamma_{M} \cap \Gamma_{L} \subseteq \Gamma_{C}$. Let $L' \vDash \tp(L/Ck_{M}\Gamma_{M})$ and let $\sigma:L \rightarrow L'$ be a partial elementary map fixing $Ck_{M}\Gamma_{M}$, as the hypothesis of Proposition \ref{isoextended} are satisfied, we can find an automorphism $\tau$ of $\mathfrak{M}$ fixing $M$ and extending $\sigma$. The map $\tau$ must be elementary by Theorem \ref{QERACF}, because its restriction to the value group is a partial elementary map of $\Gamma_{\mathfrak{M}}$. We conclude that $\tp(L/Ck_{M}\Gamma_{M}) \vdash \tp(L/M)$, as required. 
\end{proof}
In the following remark we indicate how forking independence relates to the notions of independence required in Corollary \ref{racf}. 
\begin{remark}\label{f1} Let $T$ be some $\mathcal{L}_{\val}^{*}$ complete extension of the first order theory of henselian valued fields of equicharacteristic zero with residue field algebraically closed. Let $\mathfrak{M}$ be its monster model, and $L$, $M$ substructures. Let $C \subseteq L \cap M$ be a common subfield, then if $k_{M}\Gamma_{M} \ind_{C} k_{L}\Gamma_{L}$ we have $k_{M} \ind^{alg}_{k_{C}}k_{L}$ and $\Gamma_{M} \ind^{s}_{\Gamma_{C}} \Gamma_{L}$. 
\end{remark}
\begin{proof}
By Corollary \ref{seracf}, the residue field and the value group are stably embedded and orthogonal to each other, hence $k_{M} \ind_{k_{C}}k_{L}$ and $\Gamma_{M} \ind_{\Gamma_{C}} \Gamma_{L}$.  Forking independence in the residue field implies in particular algebraic independence, so $k_{M} \ind_{k_{C}}^{alg}k_{L}$. Forking independence in the value group guarantees forking independence in the reduct to $\mathcal{L}_{AG}$ so $\Gamma_{M} \ind_{\Gamma_{C}}^{s}\Gamma_{L}$. 
\end{proof}

\subsubsection{Domination by the residue sort and the value group in the language $\mathcal{L}$.}
In this subsection, we let $T$ be some complete extension of the $\mathcal{L}$-theory of henselian valued fields of equicharacteristic zero and we let $\mathfrak{C}$ be its monster model. Given $L$ an $\mathcal{L}$ substructure of $\mathfrak{C}$ and $n \in \mathbb{N}$ we set $(\mathcal{A}_{n})_{L}=\{ \res^{n}(l) \ | \ l \in L\}$, and $\mathcal{A}_{L}=((\mathcal{A}_{n})_{L} \ | \ n \in \mathbb{N})$. The main goal of this subsection is proving that the type of a valued field over a maximal model $C$ is dominated by the power residue sorts and the value group.\\
\begin{theorem}\label{domination1} Let $L$ and $M$ be substructures of $\mathfrak{C}$, and let $C$ be a maximal model of $T$ which is also a common substructure of $L$ and $M$. If 
\begin{enumerate}
    \item $\Gamma_{L} \ind_{\Gamma_{C}}^{s} \Gamma_{M}$, 
    \item $k_{M}$ and $k_{L}$ are linearly disjoint over $k_{C}$.
\end{enumerate}
Then $\tp(L/ C\mathcal{A}_M\Gamma_{M}) \vdash \tp(L/M)$. 
\end{theorem}
\begin{proof}
Let $L^{'} \vDash \tp(L/ C\mathcal{A}_{M} \Gamma_{M})$ and let $\sigma$ be a partial elementary map sending $L$ to $L'$ fixing $C\mathcal{A}_{M} \Gamma_{M}$. By Fact \ref{intersection}, $\Gamma_{L} \cap \Gamma_{M} \subseteq \dcl(\Gamma_{C})=\Gamma_{C}$ because $C$ is definably closed. By Proposition \ref{isoextended2}, there is an $\mathcal{L}_{\val}$ valued field isomorphism $\tau:C(L,M) \rightarrow C(L', M)$ which is the identity on $M$ and $\tau \upharpoonright_{L}=\sigma$. By Proposition \ref{linearlydisjoint},  $L$ and $M$ are linearly disjoint over $C$, $L$ has the good basis property over $M$, the value group $\Gamma_{C(L,M)}$ is the group generated by $\Gamma_{M}$ and $\Gamma_{L}$ over $\Gamma_{C}$ and the residue field $k_{C(L,M)}$ is the field generated by $k_{L}$ and $k_{M}$ over $k_{C}$. In particular, any element $x \in C[L,M]$ can be represented as $\displaystyle{x= \sum_{i \leq n} l_{i}m_{i}}$ and $\displaystyle{v(x)=v(\sum_{i \leq n} l_{i}m_{i})= \min\{ v(l_{i})+v(m_{i}) \ | \ i \leq n\}}$. As $\tau$ is an $\mathcal{L}_{val}$-isomorphism, we have:
\begin{align*}
\tau(v(x))= \tau \big( \min\{v(l_{i})+v(m_{i}) \ | \ i \leq n \}\big)= \min \{ \sigma(v(l_{i}))+ v(m_{i}) \ | \ i \leq n\}= v(\tau(x)),
\end{align*}
and because $\sigma:L \rightarrow L^{'}$ is a partial elementary map fixing $\Gamma_{M}$, the restriction map $\tau: \Gamma_{C(L,M)} \rightarrow \Gamma_{C(L', M)}$ is a partial elementary map. We want to extend the $\mathcal{L}_{\val}$-isomorphism to a $\mathcal{L}$-isomorphism, we start by proving the following claim. 

\begin{claim}\label{mari1}{Fix $n \in \mathbb{N}$ and $x \in C(L,M)$ be such that $v(x) \in n\Gamma$. Then there are $a \in \mathcal{O}_{C(L,M)}^{\times}$, $l\in L$ and $m \in M$ such that $\res^{n}(x)= \pi_{n}(\res(a)) \res^{n}(l) \res^{n}(m)$.}\end{claim}
\begin{proof}
Let $x \in C(L,M)$, as $L$ has the separated basis property over $M$ there are $l' \in L$ and $m' \in M$ such that $v(x)=v(l')+v(m')$. Let $\phi(x,y)= \exists \gamma (x+y=n\gamma)$, because $\vDash \phi(v(l'),v(m'))$ the $\mathcal{L}_{AG}$-formula $\phi(x,y)$ is represented in the type $\tp(v(l')/\Gamma_{M})$. Because $\Gamma_{L} \ind_{\Gamma_{C}}^{s}\Gamma_{M}$, by Fact \ref{nonforkingheir}  $\tp_{\mathcal{L}_{AG}}(v(l')/\Gamma_{M})$ is a heir extension of $\tp_{\mathcal{L}_{AG}}(v(l')/\Gamma_{C})$ so we can find some element $c \in C$ such that $\vDash \phi(v(l'),v(c))$. Take $l=l'c \in L$ and $m=m'c^{-1}$, then $v(x)=v(l)+v(m)$ where $v(l),v(m)\in n\Gamma$. Let $a= \frac{x}{lm}$, so $x=a (lm)$ and 
\begin{align*}
\res^{n}(x)= \res^{n}(a) \res^{n}(lm)= \pi_{n}(\res(a)) \res^{n}(l)\res^{n}(m), \ \text{as desired.}
\end{align*}
\end{proof}
Let $x,y \in C(L,M)$, $n \in \mathbb{N}$ be such that $\res^{n}(x)=\res^{n}(y)$, $a,b \in \mathcal{O}_{C(L,M)}^{\times}$ and $l_{1},l_{2} \in L$ and $m_{1},m_{2}\in M$ satisfying $\res^{n}(x)= \pi_{n}(\res(a)) \res^{n}(l_{1})\res^{n}(m_{1})$ and $\res^{n}(y)= \pi_{n}(\res(b)) \res^{n}(l_{2})\res^{n}(m_{2})$. Thus:
\begin{align*}
\res^{n}(x)=\res^{n}(y) \ \text{if and only if} \ \pi_{n}(\res(a))\res^{n}(l_{1})\res^{n}(m_{1})= \pi_{n}(\res(b)) \res^{n}(l_{2})\res^{n}(m_{2}).
\end{align*}
By Fact \ref{residueok} the equality $\pi_{n}(\res(a))\res^{n}(l_{1})\res^{n}(m_{1})= \pi_{n}(\res(b)) \res^{n}(l_{2})\res^{n}(m_{2})$ can be expressed by a formula in $\tp(L/C\mathcal{A}_{M}\Gamma_{M})$, as $\sigma$ is an elementary map the same formula holds for the elements in $\sigma(L)$. As a result, $\pi_{n}(\res(\tau(a)))\res^{n}(\tau(l_{1}))\res^{n}(m_{1})= \pi_{n}(\res(\tau(b))) \res^{n}(\tau(l_{2}))\res^{n}(m_{2})$ so $\res^{n}(\tau(x))=\res^{n}(\tau(y))$.\\
Hence we can naturally extend the $\mathcal{L}_{\val}$-isomorphism $\tau$ to an $\mathcal{L}$-isomorphism, by taking maps $\tau_{n}:(\mathcal{A}_{n})_{C(L,M)} \rightarrow (\mathcal{A}_{n})_{C(L',M)}$  sending the residue class $\res^{n}(x)$ to $\res^{n}(\tau(x))$. Then $\mathbf{t}= \tau \cup \{ \tau_{n} \ | \ \ n \in \mathbb{N} \}$  is a $\mathcal{L}$-isomorphism from $C(L,M)$ into $C(L',M)$ satisfying the following properties:
\begin{enumerate}
\item $\mathbf{t}\upharpoonright _{M}=\id_{M}$ and $\mathbf{\tau} \upharpoonright_{L}=\sigma$, 
\item $\mathbf{t}:\Gamma_{C(L,M)} \rightarrow \Gamma_{C(L',M)}$ is a partial elementary map in $\Gamma_{\mathfrak{C}}$ because $\Gamma_{C(L,M)}$ is the value group generated by $\Gamma_{M}$ and $\Gamma_{L}$ over $\Gamma_{C}$ and $\sigma$ fixes $\Gamma_{M}$,
\item $\mathbf{t}: \mathcal{A}_{(C(L,M))} \rightarrow \mathcal{A}_{(C(L',M))}$ is a partial elementary map in $\mathcal{A}_{\mathfrak{C}}$.This follows by the fact that $\sigma$ is a partial elementary map fixing $\mathcal{A}_{M}$ combined with Claim \ref{mari1} and Fact \ref{residueok}.
\end{enumerate}
By Theorem \ref{QE} $\mathbf{\tau}$ is a partial elementary map and therefore can be extended to an automorphism of $\mathfrak{C}$. As a result $\tp(L/M)=\tp(L'/M)$ as required. 
\end{proof}
We conclude this section by restating our result in terms of domination. 
\begin{corollary}\label{domvalresL} Let $C \subseteq L$ be substructures of $\mathfrak{C}$ with $C$ a maximal model of $T$. Then $\tp(L/C)$ is dominated by the value group and the power residue sorts, this is for any field extension $C \subseteq M$ such that $k_{M} \ind_{k_{C}}^{alg} k_{L}$ and $\Gamma_{M} \ind_{\Gamma_{C}}^{s} \Gamma_{L}$ we have $\tp(L/C\mathcal{A}_{M}\Gamma_{M}) \vdash \tp(L/M)$
\end{corollary}
\begin{proof}
We want to show that $\tp(L/C\mathcal{A}_{M}\Gamma_{M}) \vdash \tp(L/M)$. Because $C$ is a model and the residue field is of characteristic zero, $k_{L}$ is a regular extension of $k_{C}$. By hypothesis $k_{M} \ind_{k_{C}}^{alg}k_{L}$, by Fact \ref{mago} $k_{L}$ and $k_{M}$ must be linearly disjoint over $k_{C}$. By Fact \ref{maximal}, $L$ has the good separated basis property over $C$. Hence, the hypothesis of Theorem \ref{domination1} are satisfied, so $\tp(L/C\mathcal{A}_{M}\Gamma_{M}) \vdash \tp(L/M)$ as required. 
\end{proof}
The following remark emphasizes how forking independence relates to the required independence conditions in Corollary \ref{domvalresL}. 
\begin{remark}\label{f2} Let $C$ and $L$ be as in Corollary  \ref{domvalresL}. Let $C \subseteq M$ be a field extension, such that $\mathcal{A}_{M} \Gamma_{M} \ind_{\mathcal{A}_{C}\Gamma_{C}} \mathcal{A}_{L} \Gamma_{L}$ then $k_{M} \ind^{alg}_{k_{C}} k_{L}$ and $\Gamma_{M} \ind_{\Gamma_{C}}^{s}\Gamma_{L}$. 
\end{remark}
\begin{proof}
Let $M \supseteq C$ be another structure such that $\mathcal{A}_{M} \Gamma_{M} \ind_{\mathcal{A}_{C}\Gamma_{C}} \mathcal{A}_{L} \Gamma_{L}$, because the sorts $\mathcal{A}$ and $\Gamma$ are orthogonal and purely stably embedded this is equivalent to $\mathcal{A}_{M} \ind_{\mathcal{A}_{C}} \mathcal{A}_{L}$ and $\Gamma_{L} \ind_{\Gamma_{C}} \Gamma_{M}$. In particular in the reduct to $\mathcal{L}_{AG}$ it must be the case that  $\Gamma_{M} \ind_{\Gamma_{C}}^{s} \Gamma_{L}$. Because $\mathcal{A}_{L} \ind_{\mathcal{A}_{C}} \mathcal{A}_{M}$, in particular $k_{L}$ and $k_{M}$ are algebraically independent over $k_{C}$.
\end{proof}
\subsubsection{Domination by the residue field and the value group in the $\mathcal{L}_{\ac}$-language}
In this subsection we prove a domination result for henselian valued fields of equicharacteristic zero in the language $\mathcal{L}_{\ac}$, using Theorem \ref{Pas}.\\
Adding an angular component simplifies significantly the henselian valued field, in fact it corresponds to having the exact sequence $\displaystyle{1 \rightarrow k^{\times} \rightarrow RV \rightarrow \Gamma \rightarrow 0}$ to split. However, it should be noted that adding an angular component increases the set of definable sets, so it is interesting to understand as well domination results in this framework by its own sake. \\
 Through this section $T$ is some complete extension of the $\mathcal{L}_{\ac}$-theory of henselian valued fields of equicharacteristic zero and $\mathfrak{C}$ is the monster model. Given $M$ a substructure of $\mathfrak{C}$, we will write $k(M)$ to denote $\dcl(M) \cap k_{\mathfrak{C}}$ and we observe that $\ac(M) \subseteq k(M)$. 

\begin{theorem}\label{domac} Let $L$ and $M$ be good substructures of $\mathfrak{C}$, and let $C$ be a maximal model of $T$ which is also a common substructure of $L$ and $M$. If the following conditions hold
\begin{enumerate}
    \item $k_{M}$ and $k_{L}$ are linearly disjoint over $k_{C}$,
    \item $\Gamma_{M} \cap \Gamma_{L}=\Gamma_{C}$, 
    \item $M$ or $L$ have the good separated basis property over $C$,
\end{enumerate}
then $\tp(L/ Ck(M)\Gamma_{M}) \vdash \tp(L/M)$. 
\end{theorem}
\begin{proof}
As in Theorem \ref{domination1} we start taking $L^{'} \vDash \tp(L/ Ck(M)\Gamma_{M})$, and $\sigma$ a partial elementary map sending $L$ to $L'$ fixing $Ck(M) \Gamma_{M}$. By Proposition \ref{isoextended2}, there is a $\mathcal{L}_{\val}$ valued field isomorphism $\tau: C(L,M) \rightarrow C(L', M)$ which is the identity on $M$ and $\tau \upharpoonright_{L}=\sigma$.\\
By Proposition \ref{linearlydisjoint},  $L$ and $M$ are linearly disjoint over $C$, $M$ (or $L$) has the separated basis property over $L$ (or $M$), the value group $\Gamma_{C(L,M)}$ is the group generated by $\Gamma_{M}$ and $\Gamma_{L}$ over $\Gamma_{C}$ and the residue field $k_{C(L,M)}$ is the field generated by $k_{L}$ and $k_{M}$ over $k_{C}$. In particular, any element $x \in C[L,M]$ can be represented as $\displaystyle{x= \sum_{i \leq n} l_{i}m_{i}}$ and $\displaystyle{v(x)=v(\sum_{i \leq n} l_{i}m_{i})= \min\{ v(l_{i})+v(m_{i}) \ | \ i \leq n\}}$. As $\tau$ is an $\mathcal{L}_{\val}$-isomorphism, we have:
\begin{align*}
\tau(v(x))= \tau \big( \min\{v(l_{i})+v(m_{i}) \ | \ i \leq n \}\big)= \min \{ \sigma(v(l_{i}))+ v(m_{i}) \ | \ i \leq n\}= v(\tau(x)).
\end{align*}
We want to extend the $\mathcal{L}_{\val}$-isomorphism to a $\mathcal{L}_{\ac}$-isomorphism, so it is sufficient to verify that $\tau$ respects also the angular component map. \\

\begin{claim}\label{mari2}{ Let $x \in C[L,M]$ then there are $a \in \mathcal{O}_{C(L,M)}^{\times}$, $l \in L$ and $m \in M$ such that $x=alm$. In particular,  $\tau(\ac(x))=\ac(\tau(x))$ and $\ac(x)=\res(a) \ac(l)\ac(m)$.}\end{claim}
\begin{proof}
Let  $x \in C[L,M]$ and suppose that $\displaystyle{x=\sum_{i \leq n} l_{i}m_{i}}$. Because $M$ (or $L$) has the separated basis property over $L$ (or $M$) there is some $i_{0}\leq n$ such that $v(x)=v(l_{i_{0}} m_{i_{0}})$. Let $a= \frac{x}{l_{i_{0}}{m_{0}}} \in \mathcal{O}_{C(L,M)}^{\times}$, then 
\begin{equation*}
\ac(x)= \ac(l_{i_{0}}m_{i_{0}} a)= \ac(l_{i_{0}}) \ac(m_{i_{0}})\ac(a) =\ac(l_{i_{0}}) \ac(m_{i_{0}}) \res(a).
\end{equation*}
Note that $\tau(x)= \tau(a) \sigma(l_{i_{0}})m_{i_{0}}$. Thus:
\begin{align*}
\tau(\ac(x))&= \tau(\ac(l_{i_{0}}) \ac(m_{i_{0}}) \res(a))= \tau(\ac(l_{i_{0}})) \tau(\ac(m_{i_{0}})) \tau(\res(a))\\
&= \ac(\sigma(l_{i_{0}}))\ac(m_{i_{0}}) \res(\tau(a))= \ac(\sigma(l_{i_{0}})m_{i_{0}}\tau(a)))= \ac(\tau(x)), \ \text{as required. }
\end{align*}
\end{proof}
We conclude that $\tau$ is an $\mathcal{L}_{\ac}$-isomorphism, and because $\Gamma_{C(L,M)}$ is the group generated by $\Gamma_{L}$ and $\Gamma_{M}$ over $\Gamma_{C}$  and $\sigma$ fixes $\Gamma_{M}$,  then $\tau \upharpoonright \Gamma_{C(L,M)} \rightarrow \Gamma_{C(L',M)}$ in an elementary map in $\Gamma$. Combining Claim \ref{mari2}, the fact that the residue field $k_{C(L,M)}$ is the field generated by $k_{L}$ and $k_{M}$ over $k_{C}$ and $\ac(M) \subseteq k(M)$ is fixed by $\sigma$, we can conclude that $\tau\upharpoonright \ac(C(L,M))$ is an elementary map in $\mathbf{k}$. By Theorem \ref{Pas}, such map must be elementary. 
\end{proof}
We restate our result in terms of domination, and we highlight that we required weaker hypothesis compare to Corollary \ref{domvalresL}.
\begin{corollary}\label{domvalresac} Let $T$ be some complete extension of the $\mathcal{L}_{\ac}^{*}$ first order theory of henselian valued fields of equicharacteristic zero and let $\mathfrak{C}$ be its monster model. Let $C \subseteq L$ be substructures of $\mathfrak{C}$, with $C$ maximal, $k_{L}$ a regular extension of $k_{C}$ and $\Gamma_{L}/\Gamma_{C}$ torsion free. Then $\tp(L/C)$ is dominated by the value group and the residue field, this is for any field extension $C \subseteq M$ if $k_{M} \ind_{k_{C}}^{alg} k_{L}$ and $\Gamma_{M} \ind_{\Gamma_{C}}^{s}\Gamma_{L}$ then $\tp(L/ Ck_{M}\Gamma_{M}) \vdash \tp(L/M)$. 
\end{corollary}
\begin{proof} The argument follows in a very similar manner as Corollary \ref{racf}. 
\end{proof}
\begin{remark}
\begin{enumerate}
\item As in Remark \ref{f1}, using the purely stable embeddeness and orthogonality between the residue field and the value group one can obtain that forking independence implies the independence conditions required in \ref{domvalresac}, this is if $k_{M}\Gamma_{M} \ind_{C} k_{L}\Gamma_{L}$ implies that $k_{M}\ind_{k_{C}}^{alg}k_{L}$ and $\Gamma_{M} \ind_{\Gamma_{C}}^{s}\Gamma_{L}$. 
\item A similar version of Corollary \ref{domvalresac} can be obtained for the language $\mathcal{L}_{\ac}$ without adding the constants and requiring $C$ to be a model of $T$. The proof is similar to Corollary \ref{domvalresL}.
\end{enumerate}
\end{remark}
%\end{document}

\section{Forking over maximal models in $NTP_{2}$ henselian valued fields}\label{fork}

In this section we apply the domination results obtained in Section \ref{dominationvalres} to show that forking independence over maximal models is controlled by Shelah's imaginary expansion of the value group and Shelah's imaginary expansion of the residue field in the class of henselian valued fields of equicharacteristic zero which are $NTP_{2}$. \\
We start by introducing some notation:
\begin{notation} 
\begin{enumerate}
\item Through this section we will work with a slight refinement of the languages introduced in Subsection \ref{lingue}. We denote as $\mathcal{L}'$, $\mathcal{L}_{\ac}'$ and $\mathcal{L}_{\val}'$ the extension of $\mathcal{L}$, $\mathcal{L}_{\ac}$ and $\mathcal{L}_{\val}$ (respectively); where the residue field is equipped with the multi-sorted Shelah's imaginary expansion $k^{eq}$ as well as the value group is endorsed with the language of $\Gamma^{eq}$. Likewise
\item Let $T$ be a complete first order theory in the language $\mathcal{L}'$. Given $\mathcal{S}$ a family of $\mathcal{L}'$ sorts and $A$ a set of parameters, we write $\mathcal{S}(A)$ to denote $\dcl(A)\cap \mathcal{S}$.
\end{enumerate}
\end{notation}
 It is well known that in general forking and dividing are different notions, however, they do coincide in a very large class of theories (sometimes over arbitrary sets of parameters or only over models). In \cite{NTP2} Chernikov and Kaplan shown that if a theory is $NTP_{2}$ then forking and dividing over models are the same. The following is a folklore fact, and it is the left-transitivity of dividing in any theory. 
\begin{fact}\label{left} Let $T$ be a complete first order theory and $\mathfrak{M}$ its monster model. Let $C \subseteq \mathfrak{M}$ be a set of parameters, $a,b,d \in \mathfrak{M}$, if $d \ind_{C}^{d} b$ and $a \ind_{Cd}^{d} b$ then $ad \ind_{C}^{d}b$.  
\end{fact}
 
 In \cite[Theorem 7.6]{Chernikov} Chernikov proved  that a henselian valued field of equicharacteristic zero in the $\mathcal{L}_{\ac}$ language is $NTP_{2}$ if and only if its residue field is $NTP_{2}$. Later in \cite[Theorem 3.11]{Touchard} P. Touchard proved that if $\mathcal{K}=(K, \mathbf{k},\Gamma, \ac,res,v)$ is a henselian valued field of equicharacteristic zero then $\bdn(\mathcal{K}_{\ac})=\bdn(\mathbf{k})+\bdn(\Gamma)$, where $\bdn(X)$ is the burden of the definable set $X$ as defined in \cite[Definition 1.12]{Touchard}. He also showed that if a valued field of equicharacteristic zero is considered in the language $\mathcal{L}$ then $\bdn(\mathcal{K})=\max_{n\geq 0} (\bdn(k^{\times}/(k^{\times})^{n})+ \bdn(n\Gamma))$, therefore a henselian valued field of equicharacteristic zero  is $NTP_{2}$ if and only if its residue field is $NTP_{2}$.\\

\begin{lemma}\label{forkingresval} Let $C$ be some set of parameters and $a,b$ tuples in the main field sort. 
\begin{enumerate}
\item In the $\mathcal{L}'$ language,  $\mathcal{A}(Ca) \Gamma(Ca) \ind_{C} b \ \text{if and only if} \ \mathcal{A}(Ca) \Gamma(Ca) \ind_{C} \mathcal{A}^{eq}(Cb) \Gamma^{eq}(Cb)$. 
\item In the $\mathcal{L}_{\ac}'$ language, $k(Ca) \Gamma(Ca) \ind_{C} b \ \text{if and only if} \ k(Ca) \Gamma(Ca) \ind_{C} k^{eq}(Cb) \Gamma^{eq}(Cb)$.
\item For the theory of henselian valued fields of equicharacteristic zero with residue field algebraically closed, in the $\mathcal{L}_{\val}'$-language $k(Ca) \Gamma(Ca) \ind_{C} b \ \text{if and only if} \ k(Ca) \Gamma(Ca) \ind_{C} k(Cb) \Gamma^{eq}(Cb)$.
\end{enumerate}
\end{lemma}
\begin{proof}
%We only prove the first statement, as the second one follows by a similar argument. In fact what is required in this prove is the orthogonality of the sorts $\mathbf{k}$ and $\Gamma$, and $\mathcal{A}$ and $\Gamma$ and the fact that each of them are stably embedded. This is given by Corollary \ref{esac} and \ref{stL} respectively.\\

We start by proving the first statement. The left to right direction is clear, because if $\mathcal{A}(Ca) \Gamma(Ca) \ind_{C} b$ then  $\mathcal{A}(Ca) \Gamma(Ca) \ind_{C} \acl(Cb)$, and $\mathcal{A}^{eq}(Cb) \Gamma^{eq}(Cb) \subseteq \acl(Cb)$. We proceed to prove the converse.\\
 Suppose that $\displaystyle{\mathcal{A}(Ca) \Gamma(Ca) \ind_{C} \mathcal{A}^{eq}(Cb) \Gamma^{eq}(Cb)}$, because $\mathcal{A}$ and $\Gamma$ are orthogonal to each other, this is equivalent to $\mathcal{A}(Ca) \ind_{C} \mathcal{A}^{eq}(Cb)$ and $\Gamma(Ca) \ind_{C} \Gamma^{eq}(Cb)$. By Corollary \ref{stL} $\Gamma$ and $\mathcal{A}$ are orthogonal to each other and are purely stably embedded, thus $\mathcal{A}(Ca) \Gamma(Ca) \ind_{C} b$ if and only if $\mathcal{A}(Ca) \ind_{C} b$ and $\Gamma(Ca) \ind_{C} b$. Hence, it is sufficient to prove that $\mathcal{A}(Ca) \ind_{C}b$ and $\Gamma(Ca) \ind_{C}b$.\\
\begin{claim} {$\mathcal{A}(Ca) \ind_{C}b$ and $\Gamma(Ca) \ind_{C} b$.}\end{claim}
We proceed by contradiction, and we assume that $\tp(\mathcal{A}(Ca)/ Cb)$ forks over $C$ to show that $\mathcal{A}(Ca) \not \ind_{C} \mathcal{A}^{eq}(Cb)$. We can find a formula $\phi(\bar{x},b) \in \tp(\mathcal{A}(Ca)/Cb)$, and finite set of formulas $\{ \psi_{i}(\bar{x},d_{i}) \ | \ i \leq n\}$ such that $\displaystyle{\phi(\bar{x},b) \vdash \bigvee_{i \leq l} \psi_{i}(\bar{x},d_{i})}$, where each formula $\psi_{i}(\bar{x},d_{i})$ divides over $C$.\\
As $\mathcal{A}$ is purely stably embedded, the subset of $\mathcal{A}^{n}$ defined by $\phi(\bar{x},b)$ is also defined by a formula $\eta(\bar{x}, e)$ where $e$ is a tuple of elements in $\mathcal{A}^{eq}(Cb)$. By a similar argument, the set defined by each formula $\psi_{i}(\bar{x},d_{i})$ is also defined by a formula $\epsilon_{i}(\bar{x}, f_{i})$ where $f_{i}$ is a tuple of elements in $\mathcal{A}^{eq}(Cd_{i})$. Because $\phi(\bar{x}, b)$,  $\eta(\bar{x}, e)$ define the same set, as $\psi_{i}(\bar{x},d_{i})$ and $\epsilon_{i}(\bar{x}, f_{i})$ do, then it is also the case that  $\displaystyle{\eta(\bar{x}, e) \vdash \bigvee_{i \leq l} \epsilon_{i}(\bar{x}, f_{i})}$. Since $\eta(\bar{x},e) \in \tp(\mathcal{A(Ca)}/ \mathcal{A}^{eq}(Cb))$ it is sufficient to argue that $\epsilon_{i}(\bar{x},f_{i})$ also divides over $C$.\\
Each formula $\psi_{i}(\bar{x}, d_{i})$ divides over $C$, so we can find an infinite sequence $\langle b_{j} \ | \ j < \omega \rangle$ in the type $\tp(d_{i}/C)$ such that $b_{0}=d_{i}$ and $\{ \psi_{i}(\bar{x}, b_{j}) \ | \ j < \omega \}$ is $m_{i}$-inconsistent. Let $\sigma_{j}$ be an automorphism of the monster model sending $b_{0}$ to $b_{j}$ and fixing $C$. Let $g_{j}= \sigma_{j}(f_{i})$, then $g_{0}=f_{i}$, $\langle g_{j} \ | \ j < \omega \rangle$ is in the type $\tp(f_{i}/C)$. As a result, $\{ \epsilon_{i}(\bar{x}, g_{j}) \ | \ j < \omega\}$ is also $m_{i}$-inconsistent, because $\psi_{i}(\bar{x}, b_{j})$ and $\epsilon_{i}(\bar{x}, g_{j})$ define the same subset of $\mathcal{A}^{n}$. Consequently, each $\epsilon_{i}(\bar{x},f_{i})$ divides over $C$, so $\mathcal{A}(Ca) \not \ind_{C} \mathcal{A}^{eq}(Cb)$. We conclude that if $\mathcal{A}(Ca) \ind_{C} \mathcal{A}^{eq}(Cb)$ then $\mathcal{A}(Ca) \ind_{C} b$. Likewise, one can show that if $\Gamma(Ca) \ind_{C} \Gamma^{eq}(Cb)$ then $\Gamma(Ca) \ind_{C} b$. This concludes the proof of the right to left direction. \\
Likewise, we can conclude similarly the second and the third statement. In fact, the proof only requires that the residue field and the value group are orthogonal to each other and are purely stably embedded. This is guaranteed by Corollary \ref{seac} and  \ref{seracf} respectively. 
\end{proof}
%In the following statements we will work with $T$ the $\mathcal{L}$- first order theory of the henselian valued fields of equicharacteristic zero, whose residue field is $NTP_{2}$. We fix $\mathfrak{C}$ its monster model. 
\begin{theorem} \label{forking} Let $C$ be some maximal model and assume that the residue field is $NTP_{2}$. 
\begin{enumerate}
\item In the $\mathcal{L}'$-language, $a \ind_{C}b \ \text{if and only if} \ \mathcal{A}(Ca) \Gamma(Ca) \ind_{C} \mathcal{A}^{eq}(Cb) \Gamma^{eq}(Cb)$.
\item In the $\mathcal{L}_{\ac}'$-language, $a \ind_{C} b \ \text{if and only if} \  k(Ca) \Gamma(Ca) \ind_{C} k^{eq}(Cb) \Gamma^{eq}(Cb)$.
 \end{enumerate}
\end{theorem}
\begin{proof} We start proving the first statement, the left to right direction is clear. We assume that \\
$\mathcal{A}(Ca) \Gamma(Ca) \ind_{C} \mathcal{A}^{eq}(Cb) \Gamma^{eq}(Cb)$. By Corollary \ref{stL} $\mathcal{A}$ and $\Gamma$ are purely stably embedded and orthogonal to each other, so this is equivalent to $\mathcal{A}(Ca) \ind_{C} \mathcal{A}^{eq}(Cb)$ and $\Gamma(Ca) \ind_{C} \Gamma^{eq}(Cb)$. Because $\mathcal{A}(Ca) \ind_{C} \mathcal{A}^{eq}(Cb)$, then $k(Ca) \ind_{k_{C}}^{alg} k(Cb)$. As $C$ is a model and the residue field is of characteristic zero,  $k(Ca)$ is a regular extension of $k_{C}$ so we can apply Fact \ref{mago} to conclude that $k(Ca)$ and $k(Cb)$ are linearly disjoint over $k_{C}$. Because $\Gamma(Ca) \ind_{C} (Cb)$, then $\Gamma(Ca) \ind_{\Gamma_{C}}^{s} \Gamma(Cb)$, as $\Gamma$ is a stably embedded sort and we are considering the reduct to the language of abelian groups.\\
\begin{claim}\label{arreglo} $a \ind_{C\mathcal{A}(Ca)\Gamma(Ca)}^{d} b$.\end{claim}
\begin{proof}
Let $p(x, C\mathcal{A}(Ca)\Gamma(Ca))= \tp(a/ C\mathcal{A}(Ca)\Gamma(Ca),b)$. It is sufficient to argue that no formula $\phi(\bar{x},b) \in p(x, C\mathcal{A}(Ca)\Gamma(Ca),b)$ divides over  $C\mathcal{A}(Ca)\Gamma(Ca)$. Let
 $\langle b_{i} \ | \ i \in \omega \rangle$ a $C\mathcal{A}(Ca)\Gamma(Ca)$-indiscernible sequence in the type $\tp(b/C\mathcal{A}(Ca)\Gamma(Ca))$. Let $\sigma_{i}$ be an automorphism of $\mathfrak{C}$ fixing $C\mathcal{A}(Ca)\Gamma(Ca)$ sending $b$ to $b_{i}$. By Theorem \ref{domination1} we can find an automorphism $\tau_{i}$ of $\mathfrak{C}$ which is the identity on $\dcl(Ca)$ and whose restriction to $\dcl(Cb)$ coincides with $\sigma_{i}$. In particular,
 \begin{equation*}
\vDash \phi(a,b) \leftrightarrow \vDash \phi(\tau_{i}(a), \tau_{i}(b)) \leftrightarrow \vDash \phi(a,b_{i}) \ \text{for any $i < \omega$.}
 \end{equation*}
 so $\{ \phi(x,b_{i}) \ | \ i < \omega\}$ is consistent, and we conclude that $\phi(x,b)$ does not divide over $C\mathcal{A}(Ca)\Gamma(Ca)$ as required.
 \end{proof}
 Combining Claim \ref{arreglo} with Lemma \ref{forkingresval} we have that $a \ind_{C\mathcal{A}(Ca)\Gamma(Ca)}^{d} b$ and $\mathcal{A}(Ca) \Gamma(Ca) \ind_{C}^{d} b$ so we can apply Fact \ref{left} to conclude that $a \mathcal{A}(Ca) \Gamma(Ca) \ind_{C}^{d} b$. As forking is equal to dividing over models in $NTP_{2}$ theories we have $a \mathcal{A}(Ca) \Gamma(Ca) \ind_{C} b$.
 Because $\mathcal{A}(Ca) \Gamma(Ca) \subseteq \acl(Ca)$, this is equivalent to $a \ind_{C}b$.\\
 Likewise, we can conclude the second statement for the $\mathcal{L}_{\ac}'$-language, using Theorem \ref{domac} and Corollary \ref{seac} instead. We observe that there is no need to work with the extension $\mathcal{L}_{\ac}^{*}$, as the independence assumption over the value group implies that $\Gamma(Ca) \cap \Gamma(Cb) \subseteq \Gamma(C)=\Gamma_{C}$, because $C$ is definably closed. 
\end{proof}
%AQUI MERI MER
\begin{proposition} Let $T$ be some complete extension of the $\mathcal{L}_{dp}^{*}$- first order theory of henselian valued fields of equicharacteristic zero with residue field algebraically closed and whose value group is dp-minimal. Let $\mathfrak{C}$ be its monster model and $C\subseteq \mathfrak{C}$ be a maximal field. Let $a,b \in \mathfrak{C}$ and suppose that $k(Ca)$ is a regular extension of $k_{C}$ and $\Gamma(Ca)/\Gamma_{C}$ is torsion free. We have that $a \ind_{C}b$ if and only if $k(Ca)\Gamma(Cb) \ind_{C} k(Cb)\hat{\Gamma}(Cb)$. Where $\hat{\Gamma}= \Gamma \cup \{ \Gamma/\Delta \ | \ \Delta \ \text{is a convex subgroup}\}$. 
\end{proposition}
\begin{proof}
The proof follows by a very similar argument as in Theorem \ref{forking}, applying Proposition \ref{isoextended2} instead. In fact, $\Gamma_{M} \cap \Gamma_{L} \subseteq \dcl(\Gamma_{C})= (\mathbb{Q} \otimes \Gamma_{C}) \cap \Gamma_{\mathfrak{C}}$ and as $\Gamma(Ca)/\Gamma_{C}$ is torsion free we have that $\Gamma_{M} \cap \Gamma_{L}=\Gamma_{C}$. Also, the independence in the residue field together with the assumption of $k(Ca)$ being a regular extension of $k_{C}$ guarantees that $k(Ca)$ and $k(Cb)$ are linearly disjoint over $k_{C}$. We can apply the equivalence between forking and dividing over sets of parameters in the main sort by a result of Cotter and Starchenko \cite[the remarks preceding Proposition 2.6 together with Corollary 5.6]{CS}. \\
\end{proof}
%\end{document}

\section{Domination by the internal sorts to the residue field in the language $\mathcal{L}_{RV}$}\label{domint}
In this section we investigate domination of a field by the sorts internal to the residue field over the value group in the $\mathcal{L}_{RV}$- languange. We start by fixing some notation.
\begin{notation} Given a field $F$ we denote as $F^{alg}$ its field algebraic closure.
\end{notation}
Let $K$ be an $\mathcal{L}_{RV}$-structure, we will write $k_{K}$ to denote the residue field of $K$, $\Gamma_{K}$ to denote its value group and $RV_{K}=\{ rv(k) \ | \ k \in K\}$.\\ 
Any henselian valued field $K$ can be naturally embedded into a model of $ACVF$, in fact we can simply take the algebraic closure of $K$ with the unique extension of $v$ to $K^{alg}$. We denote by $\mathcal{O}$ the valuation ring of $K$ and $\mathcal{M}$ its prime ideal, while $\mathcal{O}^{alg}$ is the valuation ring of $K^{alg}$ and $\mathcal{M}^{alg}$ indicates its maximal ideal. Hence, the sort $RV_{K}$ can be naturally embedded into $RV_{K^{alg}}$, by sending the class $x(1+\mathcal{M})$ to $x(1+\mathcal{M}^{alg})$. Likewise, there is a natural embedding from the residue field of $K$ into the residue field of $K^{alg}$, where for $x \in \mathcal{O}$ we send the class $x+\mathcal{M}$ to $x+\mathcal{M}^{alg}$. \\

In \cite[Section 3.1]{HHM2} Haskell, Hrushovski and Macpherson introduced the well known \emph{geometric language $\mathcal{L}_{\mathcal{G}}$}, in which $ACVF$ eliminates imaginaries (see \cite[Theorem 1.0.1]{HHM2}). Through this section we will equip any model of $ACVF$ with the language $\mathcal{L}_{ACVF}$ extending the language of $\mathcal{L}_{\mathcal{G}}$ and a $RV$-sort.
\begin{notation} Let $T$ be the $\mathcal{L}_{RV}$-theory of henselian valued fields of equicharacteristic zero. We will denote by $\mathfrak{C}$ its monster model, which can be embedded into the monster model $\mathfrak{U}$ of $ACVF$. Through this section we will work in both theories, so we emphasize the notation that we will be using to distinguish both theories. We will simply denote as $\dcl$, $\acl$, or $\tp(A/C)$ the definable closure, algebraic closure or the type in the language $\mathcal{L}_{RV}$. While we emphasize that $\dcl_{ACVF}, \acl_{ACVF}$ or $\tp_{ACVF}$ indicate the definable closure, algebraic closure or the type in the geometric language. We recall our notation, given $S$ a stably embedded sort in the $\mathcal{L}_{RV}$- theory and $A \subseteq \mathfrak{C}$ a set of parameters we denote as $S(A)=S \cap \dcl(A)$, while if $S$ is a stably embedded sort in $ACVF$ we indicate by $S_{ACVF}(A)= S \cap \dcl_{ACVF}(A)$.
\end{notation}
\begin{definition}\label{internal}
\begin{enumerate}
\item A definable set $E$ is said to be \emph{internal} to a definable set $D$ if there is some finite set of parameters $F$ such that $E \subseteq \dcl^{eq}(F \cup D)$.
 \item A family of definable sets $\{ E_{i}\}_{i \in I}$ is said to be internal to a definable set $D$ if for each $i \in I$ we have that $E_{i}$ is internal to $D$.
\end{enumerate}
\end{definition}

\begin{definition} Let $S\subseteq \Gamma_{\mathfrak{C}}$ and $M$ be a substructure of $\mathfrak{C}$ such that $S \subseteq \Gamma_{M}$. We write 
\begin{equation*}
\kInt_{S}^{M}=\{ k_{M}\} \cup \{ RV_{M} \cap v_{RV}^{-1}(\gamma) \ | \ \gamma \in S \}.
\end{equation*}
\end{definition}
For each $\gamma \in S$, $RV_{\mathfrak{C}} \cap v_{RV}^{-1}(\gamma)$ is internal to the residue field and the parameters required to witness the internality lie in $RV_{\mathfrak{C}} \cap v_{RV}^{-1}(\gamma)$. Indeed, given $b(1+\mathcal{M}) \in RV_{\mathfrak{C}} \cap v_{RV}^{-1}(\gamma)$ the map $\displaystyle{f: \mathcal{O}^{\times}/(1+\mathcal{M}) \rightarrow RV \cap v_{RV}^{-1}(\gamma)}$ sending the element $x(1+\mathcal{M})$ to $\big(b(1+\mathcal{M})\big) \cdot \big( x (1+\mathcal{M})\big)= bx (1+\mathcal{M})$ is a bijection. In particular, for each $\gamma \in S$, $RV_{\mathfrak{C}} \cap v_{RV}^{-1}(\gamma)$ is stably embedded and so it is $\kInt_{S}^{\mathfrak{C}}$.\\
Similarly, we can consider the structures $RV_{\mathfrak{U}} \cap v_{RV}^{-1}(\gamma)$, $\kInt_{S}^{\mathfrak{U}}$ and obtain the same results in this setting.\\

In the case of $ACVF$, let $C\subseteq \mathfrak{U}$ a set of parameters and $L$ a substructure of $\mathfrak{U}$, then  $\acl_{ACVF}(\kInt_{\Gamma_{L}}^{M})$
is precisely the part of $\mathfrak{U}^{eq}$ which is internal to the residue field and contained in sets definable over $C$ and $\Gamma_{ACVF}(L)$ (see \cite[Lemma 12.9]{HHM}). In $ACVF$ the residue field is an algebraically closed field, so it has a strongly minimal theory and forking independence coincides with algebraic independence in the field theoretic sense. In \cite[Lemma 2.6.2 ]{HHM2}, Haskell, Hrushovski and Macpherson characterize the definable sets that are internal to the residue field precisely as those that are stable and stably embedded, or more precisely as those that have finite Morley rank with the induced structure. In particular, in $ACVF$ the multi-sorted structure $\acl_{ACVF}(\kInt_{\Gamma_{L}}^{M})$ is naturally equipped with a well-behaved notion of independence, which is simply forking independence in stable theories. \\
 
We will not investigate this in the more general setting of henselian valued fields of equicharacteristic zero. Instead we use the fact that any henselian valued field of equicharacteristic zero can be naturally embedded into a model of $ACVF$ and we use the well-behaved notion of independence induced there, which in our setting corresponds to independence for the quantifier free and stable formulas. 
\begin{definition} Let $L$ and $M$ be substructures of $\mathfrak{C}$ such that $\dcl(L)=L$  and $\Gamma_{L} \subseteq \Gamma_{M}$. We consider these structures embedded in the monster model $\mathfrak{U}$ of $ACVF$. Suppose that $\Gamma_{L} \subseteq \Gamma_{M}$ and let $C \subseteq L \cap M$ be a common valued field. We say that $\kInt^{L}_{\Gamma_{L}} \ind_{C\Gamma_{L}}^{qfs} \kInt_{\Gamma_{L}}^{M}$ in $\mathfrak{C}$ if and only if $\acl_{ACVF}(\kInt_{\Gamma_{L}}^{L}) \ind_{C\Gamma_{L}} \acl_{ACVF}(\kInt_{\Gamma_{L}}^{M})$ in $\mathfrak{U}$. 
\end{definition}

%\textcolor{red}{decidir que hacer con lo siguiente}\\
%In the case where $B$ is a model, any element of $RV_{\gamma}(B)$ is definable over $k(B)$ from any other element. In fact,  there is a definable bijection between $RV \cap v_{RV}^{-1}(0) \rightarrow RV \cap v_{RV}^{-1}(\gamma)$ given by sending the element $x (1+\mathcal{M}) \rightarrow bx(1+\mathcal{M})$ where $v(b)=\gamma$ and $v(x)=0$. Therefore, for each $\gamma \in S$ if we take an element $d_{\gamma} \in RV_{\gamma}(B)$ we have:
%\begin{align*}
%kInt_{S}^{B}= acl(k(B) \cup \{ d_{\gamma} \ | \ \gamma \in S\})
%\end{align*}

%Let  $A \subseteq \Gamma_{\mathfrak{C}}$, we define:\\
%$R_{A}=acl^{eq}(k_{\mathfrak{C}} \cup })$
%\end{definition}
%\begin{fact} $R_{A}$ is internal to the residue field, and therefore stably embedded. 
%\end{fact}
%\textcolor{red}{This should be at the end of chapter $1$ in HHM, look for reference. Hay que incluir tambien la interpretacion en un modelo}\\

Our next goal is showing that types over maximally complete bases are dominated by the sorts internal to the residue field, to achieve this final milestone we will need Lemma \ref{indepequiv}, which generalizes  \cite[Lemma 12.9 and 12.10]{HHM}, both obtained for algebraically closed substructures.\\
 The following is a well-known fact for valued field extensions, we use \cite{dries} as a reference. 
\begin{fact}\label{inequality} Let $C \subseteq L$ be a valued field extension, where $\mathcal{O}_{C}$ is the valuation ring of $C$ and $\mathcal{O}_{L}$ is the valuation ring of $L$. Let $(b_{i} \ | \ i \in \mathcal{I})$ be a sequence of elements of $\mathcal{O}_{L}^{\times}$ such that $\res(b_{i})$ in $\mathbf{k}_{L}$ is algebraically independent over $\mathbf{k}_{C}$. And let $(a_{j} \ | \ j \in J)$ be a family of elements of $L^{\times}$ such that the family $\big(v(a_{j}) \ | \ j \in J\big)$ in $\mathbb{Q} \otimes \Gamma_{L}$ is $\mathbb{Q}$-linearly independent over $\mathbb{Q} \otimes \Gamma_{C}$. Assume that $I \cap J=\emptyset$ and define $d_{k} \in L$ for $k \in I \cup J$ by $d_{i}={b}_{i}$ for $i \in I$ and $d_{j}=a_{j}$ for $j \in J$. Then:
\begin{enumerate}
\item $\big( d_{k} \ | \ k \in I \cup J \big)$ in $L$ is algebraically independent over $C$, and 
\item if $C \subseteq L$ is an extension of finite transcendence degree, then $\trdeg(\mathbf{k}_{L}/\mathbf{k}_{C})+\dim_{\mathbb{Q}}( \mathbb{Q} \otimes \Gamma_{L}/ \mathbb{Q} \otimes \Gamma_{C}) \leq \trdeg(L/C)$.
\end{enumerate}
The second statement is known as the \emph{Zariski-Abhyankar Inequality} and it is a direct consequence of the first one. \\
%\textcolor{red}{lo anterior no necesita el grado de trascendencia finito, pero ok asi esta para citarlo directo de Dries. }
\end{fact}
\begin{proof}
This is \cite[Lemma 3.24 and Corollary 3.25]{dries}. 
\end{proof}

\begin{fact}\label{dcl} Let $L$ be a substructure of $\mathfrak{U}$, then $\Gamma_{ACVF}(L)=\mathbb{Q} \otimes \Gamma_{L}$ and $k_{ACVF}(L)\subseteq \acl_{ACVF}(L) \cap \mathbf{k}= k_{L}^{alg}$.
\end{fact}
\begin{proof} Let $L$ be a substructure of $\mathfrak{U}$ then $L \subseteq L^{alg} \subseteq \mathfrak{U}$ and $L^{alg} \vDash ACVF$. In particular, $\Gamma_{ACVF}(L) \subseteq \Gamma_{ACVF}(L^{alg})$. It is therefore sufficient to argue that $\Gamma_{ACVF}(L^{alg})=\mathbb{Q} \otimes \Gamma_{L}$. Let $\gamma \in \Gamma_{ACVF}(L^{alg})$ and $\phi(x,l)$ be some formula witnessing that $\gamma$ is in the definable closure of $L^{alg}$, so $l \in L^{alg}$. As $\Gamma$ is purely stably embedded there is some $\mathcal{L}_{OAG}$-formula $\psi(x,a)$ such that 
\begin{equation*}
\mathfrak{U}\vDash \forall x \big((\phi(x,l) \leftrightarrow \psi(x,a) \big).
\end{equation*}
Because $L^{alg}$ is a model 
\begin{equation*}
L^{alg}\vDash \exists a \in \Gamma^{n} \forall x \big((\phi(x,l) \leftrightarrow \psi(x,a) \big).
\end{equation*}
Thus, we can find some element $a \in \Gamma_{L^{alg}}^{n}$ such that $L^{alg} \vDash \psi(\gamma, a)$ and $|\psi(L^{alg},a)|=1$. By quantifier elimination of $ODAG$, $\psi(x,a)$ must be equivalent to a formula $x=\tau(a)$, where $\tau$ is a term in the language $\mathcal{L}_{OAG}$. Thus $\gamma \in \Gamma_{L^{alg}}=\mathbb{Q}\otimes \Gamma_{L}$.\\
We continue arguing for the residue field. It is clear that $k_{ACVF}(L) \subseteq \acl_{ACVF}(L) \cap \mathbf{k}$. Thus it is sufficient to argue that $\acl_{ACVF}(L) \cap \mathbf{k}= k_{L}^{alg}$. Let $b \in \acl_{ACVF}(L) \subseteq \acl_{ACVF}(L^{alg})$ and $\phi(x,l)$ the formula witnessing that $b$ is algebraic over $L$. Because $L^{alg}$ is a model and $\mathbf{k}$ is purely stably embedded, there is some formula $\psi(x,z)$ in the language of fields such that:
\begin{equation*}
  L^{alg} \vDash \exists d \in k^{r} \big(\forall x \big( \phi(x,l) \leftrightarrow \psi(x,d)\big) \big).
\end{equation*}
Consequently, we can find some tuple $d \in k_{L^{alg}}$ such that $L^{alg} \vDash \psi(b,d)$ and $\psi(L^{alg},d)$ is finite. By quantifier elimination in $ACF$, $b \in k_{L^{alg}}=k_{L}^{alg}$. 
\end{proof}
\begin{notation}
Let $k_{C}$ be a subfield of the residue field and $\mathbf{a}=\langle a_{i} \ | \ i \leq n \rangle$ be a tuple of elements in the residue field, we denote as $k_{C}\langle a_{i} \ | \ i \leq n \rangle$ the field generated by $k_{C}$ and the tuple $\mathbf{a}$, i.e. $k_{C}\langle a_{i} \ | \ i \leq n \rangle$. 
\end{notation}
%\textcolor{red}{desde AQUI empezamos sin internet}\\
In the case of $ACVF$ the following statement is an immediate consequence of the Zariski-Abhyankar inequality.
\begin{corollary}\label{cuadra} Let $C \subseteq L$ be a valued field extension. Let $\mathcal{O}_{C}$  and $\mathcal{O}_{L}$ be the valuation rings of $C$ and $L$ respectively. Let $\langle b_{i} \ | \ i \leq r \rangle$ be a sequence of elements of $\mathcal{O}_{L}^{\times}$ such that $\res(b_{i})$ in $\mathbf{k}_{L}$ is algebraically independent over $\mathbf{k}_{C}$.  And let $\langle a_{j} \ | \ j \leq s \rangle$ be a sequence of elements of $L^{\times}$ such that $\langle v(a_{j}) \ | \ j \leq s \rangle$ in $\mathbb{Q} \otimes \Gamma_{L}$ is $\mathbb{Q}$-linearly independent over $\mathbb{Q} \otimes \Gamma_{C}$. Let $E$ be the field generated by $C$ and $\langle b_{i} \ | \ i \leq r \rangle$ and $\langle a_{j} \ | \ j \leq s \rangle$, then:
\begin{itemize}
    \item $\displaystyle{\Gamma_{ACVF}(E) \subseteq(\mathbb{Q} \otimes \Gamma_{C}) \oplus \bigoplus_{j \leq s} (\mathbb{Q} \otimes v(a_{j}))}$, and
    \item $k_{ACVF}(E) \subseteq \acl_{ACVF}(E) \cap \mathbf{k} \subseteq (k_{C}\langle \res(b_{i}) \ | \ i  \leq r \rangle)^{alg}$.
\end{itemize}
In particular, if for each $ j \leq s$ we let $d_{j}=\rv(a_{j})$ then:
\begin{itemize}
\item $\displaystyle{\Gamma_{ACVF}(Cd_{1},\dots,d_{s},\res(b_{1}),\dots,\res(b_{r})) \subseteq (\mathbb{Q} \otimes \Gamma_{C}) \oplus \bigoplus_{j \leq s} (\mathbb{Q} \otimes v(a_{j}))}$ and, 
\item $k_{ACVF}(Cd_{1},\dots,d_{s},\res(b_{1}), \dots,\res(b_{r})) \subseteq \acl_{ACVF}(Cd_{1},\dots,d_{s},\res(b_{1}), \dots,\res(b_{r})) \cap \mathbf{k} \subseteq (k_{C}\langle \res(b_{i}) \ | \ i  \leq r\rangle)^{alg}$.
\end{itemize}
\end{corollary}
\begin{proof}

Let $C \subseteq L$ and $\langle a_{j} \ | \ j \leq s\rangle$, $\langle b_{i} \ | \ i \leq r \rangle$ be tuples as in the statement. Let $E_{0}$ be the field generated by $C$ and the tuple $\langle b_{i} \ | \ i \leq r\rangle$. By Fact \ref{inequality} and Fact \ref{dcl}:
\begin{align*}
r=\trdeg(E_{0}/ C) \geq \trdeg(k_{ACVF}(E_{0})/ k_{C})+ \dim_{\mathbb{Q}}(\Gamma_{ACVF}(E_{0})/\Gamma_{ACVF}(C)), 
\end{align*}
because $\trdeg(k_{ACVF}(E_{0})/ k_{C}) \geq r$, then  $\trdeg(k_{ACVF}(E_{0})/ k_{C})=r$ and $\dim_{\mathbb{Q}}(\Gamma_{ACVF}(E_{0})/\Gamma_{ACVF}(C))=0$. In particular, $\Gamma_{ACVF}(E_{0}) \subseteq \Gamma_{ACVF}(C)=\mathbb{Q}\otimes \Gamma_{C}$, and 
\begin{equation*}
k_{ACVF}(E_{0}) \subseteq \acl_{ACVF}(E_{0})\cap \mathbf{k}=k_{E_{0}}^{alg}= (k_{C}\langle \res(b_{i}) \ | \ i \leq r\rangle)^{alg}.
\end{equation*}
Let $E$ be the field generated by $E_{0}$ and $\langle a_{i} \ | \ i \leq s\rangle$. Again by Fact \ref{inequality}, 
\begin{align*}
s=\trdeg(E/E_{0}) \geq \trdeg(k_{ACVF}(E)/ k_{ACVF}(E_{0}))+\dim_{\mathbb{Q}}(\Gamma_{ACVF}(E)/ \Gamma_{ACVF}(E_{0})).
\end{align*}
Because $\langle v(a_{i}) \ | \ i \leq s \rangle \subseteq \Gamma_{ACVF}(E)$ and $\Gamma_{ACVF}(E_{0})\subseteq \mathbb{Q} \otimes \Gamma_{C}$, then $s=\dim_{\mathbb{Q}}(\Gamma_{ACVF}(E)/ \Gamma_{ACVF}(E_{0}))$. Thus, $\trdeg(k_{ACVF}(E)/ k_{ACVF}(E_{0}))=0$. Summarizing all the above, we conclude that:
\begin{itemize}
\item $k_{ACVF} (E) \subseteq \acl_{ACVF}(E) \cap \mathbf{k}=k_{E}^{alg}=(k_{C}\langle \res(b_{i}) \ | \ i \leq r \rangle)^{alg}$, and
\item  $\displaystyle{\Gamma_{ACVF}(E) \subseteq(\mathbb{Q} \otimes \Gamma_{C}) \oplus \bigoplus_{j \leq s} (\mathbb{Q} \otimes v(a_{j}))}$, as required. 
\end{itemize}
The second part of the statement follows immediately by the fact that $\res(b_{j}) \in \dcl_{ACVF}(b_{j})$ and $d_{i}=\rv(a_{i}) \in \dcl_{ACVF}(a_{i})$. 
\end{proof}

\begin{lemma} \label{indepequiv} Let $L,M$ be substructures of $\mathfrak{U}$, the monster model of $ACVF$.  Let $C$ be a common substructure of $L$ and $M$ and suppose that $\Gamma_{L} \subseteq \Gamma_{M}$. Let
\begin{itemize}
\item $A=\{ a_{i} \ | \ i \in R\} \subseteq L$ be such that $\{ v(a_{i}) \ | \ i \in R \}$ is a maximally $\mathbb{Q}$-linearly independent set of $\Gamma_{L} \subseteq \Gamma(L)= \mathbb{Q} \otimes \Gamma(L)$ over $\Gamma_{C}$. 
\item  $E=\{ e_{i} \ | \ i \in R \} \subseteq M$ satisfying $v(e_{i})= v(a_{i})$, 
\item $B=\{ b_{j} \ | \ j \in S \} \subseteq O_{L}^{\times}$ such that $\{ \res(b_{j}) \ | \ j \in S\}$ is a transcendence base of $k_{L} \subseteq \acl(L) \cap \mathbf{k}=k_{L}^{alg}$ over $k_{C}$.
\end{itemize}
The following statements are equivalent:
\begin{enumerate}
    \item The set
    \begin{align*}
    \{\res \big( \frac{a_{i}}{e_{i}}\big), \res(b_{j}) \ | \ j\in S, i \in R\} \ \text{is algebraically independent over $k_{M}$.}
    \end{align*}
    \item The structures $\acl_{ACVF}(\kInt_{\Gamma_{L}}^{L})$, $\acl_{ACVF}(\kInt_{\Gamma_{L}}^{M})$ are independent over $C\Gamma_{L}$. 
    %$R(L)_{\Gamma_{L}}$ and $R(M)_{\Gamma_{L}}$ are algebraically independent over $\Gamma_{L}C$ (in the model theoretic sense). 
\end{enumerate}
\end{lemma}
\begin{proof}
Let $(a_{i})$, $(e_{i})$ and $\res(b_{j})$ satisfying the required hypothesis. Let $d_{i}$ be the code of the open ball $B_{v(a_{i})}(a_{i})=\{ x \in \mathfrak{U} \ | \ v(x-a_{i})>v(a_{i})\}$, note that this code is inter-definable with the class $\rv(a_{i}) \in \dcl_{ACVF}(a_{i})$. For notational convenience we will assume that $R$ and $S$ are finite and equal to $\{1, \dots, r\}$ and $\{ 1, \dots, s\}$ respectively, as the more general argument follows in a similar manner by applying the argument to any finite sequence.\\

\begin{claim}\label{vic}{The set $D=\{d_{1},\dots,d_{r},\res(b_{1}),\dots,\res(b_{s})\} \subseteq \acl_{ACVF}(\kInt^{L}_{\Gamma_{L}})$ is algebraically independent over $C\Gamma_{L}$. (In the model theoretic sense)}\end{claim}
\begin{proof}
 We proceed by contradiction and we argue by cases. Suppose the existence of some index $j_{0} \leq s$ such that $\res(b_{j_{0}}) \in \acl_{ACVF}(D_{0}C \Gamma_{L})$, where $D_{0}=D \backslash \{ \res(b_{j_{0}})\}$. Let $B_{0}= B \backslash \{b_{j_{0}}\}$ and  note that $\res(b_{j_{0}}) \in \acl_{ACVF}(D_{0} C \Gamma_{L}) \subseteq \acl_{ACVF}(C(A,B_{0}))$, where $C(A,B_{0})$ denotes the field generated over $C$ by $A$ and $B_{0}$. By Corollary \ref{cuadra},
 \begin{equation*}
  \res(b_{j_{0}}) \in \acl_{ACVF}(C(A, B_{0}))\cap \mathbf{k} \subseteq (k_{C}\langle \res(b_{j})\ | \  j \leq s, j \neq j_{0} \rangle)^{alg}.
  \end{equation*}
This contradicts the choice of the elements $(b_{i} \ | \ i \leq s)$. We now assume that for some index $i_{0} \leq s$ such that $d_{i_{0}} \in \acl_{ACVF}(CD_{0}\Gamma_{L})$ where $D_{0}=D \backslash \{ d_{i_{0}}\}$. Let $E_{0}=\acl(C(A_{0},B))$ where $A_{0}= A \backslash \{a_{i_{0}}\}$, and $C(A_{0},B)$ denotes the field generated by $A_{0}$ and $B$ over $C$. Let $G= RV_{E_{0}}$, by Corollary \ref{cuadra} and Fact \ref{dcl}:
\begin{itemize}
\item $k_{E_{0}}=\acl_{ACVF}(E_{0}) \cap \mathbf{k}=(k_{C}<\res(b_{j}) \ | \ j \leq s>)^{alg}$, and
\item $\displaystyle{\Gamma_{E_{0}}=\Gamma_{ACVF}(E_{0})=(\mathbb{Q} \otimes \Gamma_{C}) \oplus \bigoplus_{j \neq j_{0}}( \mathbb{Q} \otimes v(a_{j}))=v_{RV}(G)}$.
\end{itemize}
Moreover by construction $\displaystyle{\Gamma_{L} \subseteq (\mathbb{Q} \otimes \Gamma_{C}) \oplus \bigoplus_{j \leq r } (\mathbb{Q} \otimes v(a_{j}))}$. Let $\gamma=v(a_{i_{0}})$ and let $\phi(x,\gamma)$ be the $\mathcal{L}(CD_{0})$-formula witnessing that $d_{i_{0}} \in \acl_{ACVF}(CD_{0}\Gamma_{L})$.  Because the residue field is infinite, we can find an element $d \in RV_{\mathfrak{U}}\cap v^{-1}(\gamma)$, such that $\mathfrak{U} \vDash \neg \phi(d, \gamma)$. To simplify the notation, we write $\hat{\Gamma}$ to denote $\displaystyle{(\mathbb{Q} \otimes \Gamma_{C}) \oplus \bigoplus_{j \leq r } (\mathbb{Q} \otimes v(a_{j}))}$.\\
  By Corollary \ref{cuadra}, 
  \begin{align*}
  k_{ACVF}(CD_{0} \cup \{ d\})& \subseteq (k_{C}\langle \res(b_{j}) \ | \ j \leq s \rangle)^{alg}=k_{E_{0}} \ and\\
  k_{ACVF}(CD) &\subseteq (k_{C}\langle \res(b_{j}) \ | \ j \leq s\rangle)^{alg}=k_{E_{0}}.
  \end{align*}
Let $G_{1}= G \cdot d_{i_{0}}^{\mathbb{Z}}$ and $G_{2}= G \cdot d^{\mathbb{Z}}$ and consider the partial isomorphism:
\begin{align*}
    f:& G_{1} \rightarrow G_{2}\\
    & g d_{i_{0}}^{n} \rightarrow g d^{n}, \text{ where $n \in \mathbb{Z}$}.
\end{align*}
Let $f_{v}= id_{\hat{\Gamma}}$, and $f_{r}= id_{k_{E_{0}}}$, then the triple $(f_{r}, f, f_{v}): (k_{E_{0}},G_{1},\hat{\Gamma}) \rightarrow (k_{E_{0}}, G_{2}, \hat{\Gamma})$  is a partial isomorphism in the $\mathcal{L}_{rv}$ language [See Definition \ref{rvok}]. By Proposition \ref{RVRACF} the partial isomorphism $(f_{r},f, f_{v})$ must be an elementary map. In particular $\tp(d/ CD_{0}\Gamma_{L})=\tp(d_{i_{0}}/ CD_{0}\Gamma_{L})$, but this leads us to a contradiction because $\mathfrak{U} \vDash \neg \phi(d, \gamma) \wedge \phi(d_{i_{0}},\gamma)$. 
\end{proof}
We can now prove the equivalence between $(1)$ and $(2)$. By Claim \ref{vic} $D \subseteq \acl_{ACVF}(\kInt_{\Gamma_{L}}^{L})$ is algebraically independent (in the model theoretic sense) over $C\Gamma_{L}$. Each fiber $RV_{\mathfrak{U}} \cap v^{-1}(\gamma)$ is stably embedded and internal to the residue field (which eliminates imaginaries), so it must be a strongly minimal set. Therefore, algebraic independence in the model theoretic sense over $C\Gamma_{L}$ coincides with forking independence in the stable sense, in particular $MR(D/C\Gamma_{L})= s+r$. Thus,  $\acl_{ACVF}(\kInt^{L}_{\Gamma_{L}} )\ind_{C\Gamma_{L}} \acl_{ACVF}(\kInt^{M}_{\Gamma_{L}})$ if and only if $MR(D/ \acl_{ACVF}(\kInt_{\Gamma_{L}}^{M}))=r+s$. As $\acl_{ACVF}(\kInt_{\Gamma_{L}}^{\mathfrak{U}})$ is stably embedded, $MR(D/ \acl_{ACVF}(\kInt_{\Gamma_{L}}^{M}))=r+s$ if and only if $MR(D/ M)=r+s$. Because each element $d_{i}$ is interdefinable over $M$ with $\res\big(\frac{a_{i}}{e_{i}} \big)$, then $MR(D/ M)=r+s$ if and only if 
$MR \big(\{ \res\big( \frac{a_{1}}{e_{1}}\big), \dots, \res\big( \frac{a_{r}}{e_{r}}\big), \res(b_{1}), \dots, \res(b_{s})\} \big/M)=r+s$. As the residue field is purely stably embedded, $MR(\{ \res\big(\frac{a_{1}}{e_{1}} \big),\dots, \res\big(\frac{a_{r}}{e_{r}} \big),\res(b_{1}),\dots,\res(b_{s})\}/ M)=r+s$  if and only if $MR\big(\{ \res\big(\frac{a_{1}}{e_{1}} \big),\dots, \res\big(\frac{a_{r}}{e_{r}} \big),\res(b_{1}),\dots, \res(b_{s})\}/ k_{M} \big)=r+s$, thus $\{ \res\big(\frac{a_{1}}{e_{1}} \big),\dots, \res\big(\frac{a_{r}}{e_{r}} \big),\res(b_{1}),\dots, \res(b_{s})\}$ is algebraically independent over $k_{M}$. We conclude therefore the equivalence between $(1)$ and $(2)$ as required. 
\end{proof}

We emphasize that in the following statement we work for $T$ a complete extension of the $\mathcal{L}_{RV}$-theory of henselian valued fields of equicharacteristic zero, and we let $\mathfrak{C}$ be its monster model. The following theorem generalizes ideas present in \cite[Proposition 12.15]{HHM}, we include all details for sake of completeness.
\begin{notation}
Given a set of parameters $B \subseteq \mathbf{k}_{\mathfrak{C}}$ we will denote as $\cl(B)$ the field theoretic algebraic closure of $B$ inside of $\mathbf{k}_{\mathfrak{C}}$.\\
We recall that given a substructure $M$ of $\mathfrak{C}$ for each $n \in \mathbb{N}$ we denote as $(\mathcal{A}_{n})_{M}=\{ \res^{n}(m) \ | \ m \in M\}$ and $\mathcal{A}_{M}=((\mathcal{A}_{n})_{M} \ | \ n \in \mathbb{N})$. 
\end{notation}

\begin{theorem}\label{iso2} Let  $L$ and $M$ be substructures of  $\mathfrak{C}$ and let $C \subseteq L \cap M$ be a common substructure which is a maximal model of $T$. We suppose:
\begin{enumerate}
\item $\Gamma_{L} \subseteq \Gamma_{M}$ and $\Gamma(L)=\Gamma_{L}$,
\item $\kInt^{L}_{\Gamma_{L}} \ind_{C\Gamma_{L}}^{qfs} \kInt^{M}_{\Gamma_{L}}$,
\item $L$ has finite transcendence degree over $C$.
\end{enumerate}
Then $\tp(L/C\Gamma_{L} \mathcal{A}_{M} \kInt_{\Gamma_{L}}^{M}) \vdash \tp(L/M)$. 
\end{theorem}
\begin{proof}
Let $L' \vDash tp(L/C\Gamma_{L} \mathcal{A}_{M} \kInt_{\Gamma_{L}}^{M})$ and let $\sigma$ be an automorphism of $\mathfrak{C}$ fixing $C\Gamma_{L} \mathcal{A}_{M} \kInt_{\Gamma_{L}}^{M}$ taking $L$ to $L'$. \\

 \textit{Step $1$: Without loss of generality we may assume that $\sigma$ fixes $\Gamma_{M}$.}
 \begin{proof}
Let $\beta \in \Gamma_{M}$ such that $\sigma(\beta)=\beta'$. Because $\Gamma$ is stably embedded it is sufficient to prove that $\beta$ and $\beta'$ realize the same type over $\Gamma(L'\kInt^{M}_{\Gamma_{L}})$. So, we can take an automorphism of the structure $\tau$ fixing $L'\kInt_{\Gamma_{L}}^{M}$ sending $\beta'$ to $\beta$ and we may replace $\sigma$ by $\tau \circ \sigma$. To show that $\tp(\beta/ \Gamma(L'\kInt^{M}_{\Gamma_{L}}))=\tp(\beta'/ \Gamma(L'\kInt^{M}_{\Gamma_{L}}))$ we will argue that $\Gamma(L'\kInt^{M}_{\Gamma_{L}})=\Gamma_{L}$. Let $f$ be a $L'$-definable function from $\kInt^{M}_{\Gamma_{L}}$ to $\Gamma$. We aim to prove that for each $x \in \kInt_{\Gamma_{L}}^{M}$ we have that $f(x) \in \Gamma(L')$. 
For each $\gamma \in \Gamma_{L}$, the function $f$ takes finitely many values on $RV_{M} \cap v_{RV}^{-1}(\gamma)\subseteq RV_{\mathfrak{C}} \cap v_{RV}^{-1}(\gamma)$, because the power residue sorts and the value group are orthogonal to each other. Hence, for each $ \gamma \in \Gamma_{L}$ the set $f\big( RV_{M} \cap v^{-1}(\gamma)\big)$ is finite, thus algebraic over $L^{'}$. Then $f\big( RV_{M} \cap v^{-1}(\gamma)\big) \subseteq \acl(L') \cap \Gamma=\Gamma(L')=\Gamma_{L'}=\Gamma_{L}$, as required. 
\end{proof}
As in Proposition $12.15$ in \cite{HHM} we start by perturbing the valuation on $C(L,M)$. \\
\par \textit{Step $2$ : There is some valuation $\hat{v}$ on $C(L,M)$  finer that $v$ satisfying the following properties}
\begin{itemize}
    \item $\Gamma_{(L,\hat{v})} \cap \Gamma_{(M, \hat{v})}= \Gamma_{(C,\hat{v})}$,
    \item $k_{(L,\hat{v})}$ and $k_{(M,\hat{v})}$ are linearly disjoint over $k_{(C,\hat{v})}$,
    \item for any element $x \in M$, we have that $v(x)=\hat{v}(x)$. 
\end{itemize}
\begin{proof}
We choose elements $\{ a_{i} \ | \ i \in r\}, \{ b_{j} \ | \ j \in s\} \subseteq L$ and $\{ e_{i} \ | \ i \in r\}  \subseteq M$ satisfying the hypotheses of Lemma \ref{indepequiv}. By hypothesis, $\kInt_{\Gamma_{L}}^{L} \ind_{C\Gamma_{L}}^{qfs} \kInt_{\Gamma_{L}}^{M}$ thus $\acl_{ACVF}(\kInt_{\Gamma_{L}}^{L}) \ind_{C\Gamma_{L}} \acl_{ACVF}(\kInt_{\Gamma_{L}}^{M})$, so we can apply Lemma  \ref{indepequiv} and we obtain that: 
\begin{align*}
    \{ \res\big(\frac{a_{1}}{e_{1}} \big), \dots, \res\big(\frac{a_{r}}{e_{r}} \big), \res(b_{1}),\dots,\res(b_{s}) \}
\end{align*}
are algebraically independent (in the field theoretic sense) over $k_{M}$. For each $i \leq r$ we define:
\begin{align*}
    R^{i}= \cl(k_{M},\res\big(\frac{a_{1}}{e_{1}} \big), \dots, \res\big(\frac{a_{i}}{e_{i}} \big), \res(b_{1}),\dots,\res(b_{s})).
\end{align*}
In particular:
\begin{align*}
    R^{(0)}&= \cl(k_{M}, \res(b_{1}),\dots, \res(b_{s}))= \cl(k_{M}, k_{L}),\\
     R^{(r)}&=\cl\big(k_{M},k_{L}, \res\big(\frac{a_{1}}{e_{1}} \big), \dots, \res\big(\frac{a_{r}}{e_{r}} \big) \big).
\end{align*}
Let $p_{i}: R^{i+1} \rightarrow R^{i} \cup \{ \infty\}$ be a place which fixes $R^{i}$ and $p_{i}\big(\res\big(\frac{a_{i}}{e_{i}} \big) \big)=0$, such map can be found by the algebraic independence of
   $ \{\res\big(\frac{a_{i}}{e_{i}} \big), \dots, \res\big(\frac{a_{r}}{e_{r}} \big), \res(b_{1}),\dots, \res(b_{s}) \}$ over $k_{M}$. \\
   
   Let $p_{v}: C(L,M) \rightarrow k_{C(L,M)} \cup \{\infty\}$ be  the place corresponding to the valuation ring over $C(L,M)$ given by Proposition \ref{placeval}, and fix a place $p^{*}:k_{C(L,M)}\rightarrow R^{r} \cup \{ \infty\}$ fixing $R^{r}$.\\
We take the place $\hat{p}:C(L,M) \rightarrow \cl(k_{L},k_{M}) \cup \{ \infty\} $ given by taking the composition 
\begin{align*}
    \hat{p}=p_{0} \circ p_{1} \circ \dots \circ p_{r-1} \circ p^{*} \circ p_{v}.
\end{align*}
Let $\hat{v}$ be the valuation over $C(L,M)$ induced by $\hat{p}$, which is a refinement of the original valuation $v$. Because each of the places is the identity map on $k_{M}$, we may identify the valued field $(M,v)$ with $(M,\hat{v})$ and the value groups $\Gamma_{(M,v)}$ and $\Gamma_{(M,\hat{v})}$. So we now have two valuations $v$ and $\hat{v}$ induced over $C(L,M)$, and the construction ensures that the following conditions are satisfied:
\begin{itemize}
\item if $x \in M$, then $v(x)=\hat{v}(x)$ and if $x,y \in L$ and $v(x) \leq v(y)$ then $\hat{v}(x) \leq \hat{v}(y)$,
\item and by Lemma \ref{inf} $(2)$ for any $x \in C(L,M)$ with $v(x)>0$ we have:
\begin{align*}
0 < \hat{v} \big( \frac{a_{1}}{e_{1}}\big)<< \dots << \hat{v} \big( \frac{a_{r}}{e_{r}}\big)<< \hat{v}(x). 
\end{align*}

\end{itemize}
 Likewise, we can identify $(L,v)$ and $(L,\hat{v})$ and their value groups, as all the places are the identity map as well over $k_{L}$. However, it is impossible to identify $(M,v)$ and $(M,\hat{v})$, and $(L,v)$ with $(L,\hat{v})
$ simultaneously. We should not identify $(L,v)$ and $(L,\hat{v})$. For each $i \leq r$, let $\eta_{i}= \hat{v}\big( \frac{a_{i}}{e_{i}}\big)$ and $\hat{v}(e_{i})=v(e_{i})=\epsilon_{i}$.\\
Let 
\begin{equation*}
\Delta=\{ \hat{v}(x), -\hat{v}(x) \ | \ x \in C(L,M), p_{v}(x) \notin \{ 0, \infty\} \ \text{and} \  \hat{p}(x)=0\} \cup \{ 0_{\Gamma_{(C(L,M), \hat{v})}} \}.
\end{equation*}
By Lemma \ref{inf} $(1)$ $\Delta$ is a convex subgroup of $\Gamma_{(C(L,M), \hat{v})}$, that contains the subgroup generated by $\{ \eta_{i} \ | \ i \leq r\}$. (Because for each $i \leq r$, $p_{v}\big(\frac{a_{i}}{e_{i}}) \notin \{ 0, \infty\}$ while $\hat{p}\big(\frac{a_{i}}{e_{i}} \big)=0$).\\ 
We continue to show that the refined valuation satisfies the independence conditions over the value group and its residue field. \\

\begin{claim}{ $\Gamma_{(M,\hat{v})} \cap \Gamma_{(L,\hat{v})}=\Gamma_{(C,\hat{v})}$.} \end{claim}
\begin{proof}
Take elements $m \in M$ and $l \in L$ such that $\hat{v}(m)=\hat{v}(l)$. 
By hypothesis $\{ v(a_{i}) \ | \ i \leq r \} \subseteq \Gamma_{(L,v)}$ and it is a $\mathbb{Q}$-linearly independent set over $\Gamma_{(C,v)}$, hence $v(l)$ must belong to the pure hull of the subgroup generated by $\{ v(a_{i}) \ | \ i \leq r \}$ and $\Gamma_{(C,v)}$, thus we can find we can find $ p_{i} \in \mathbb{Z}$,$\gamma \in \Gamma_{C}$ and $k \in \mathbb{N}^{\geq 1}$ such that:
\begin{align}\label{eq1}
kv(l)=  \sum_{i=1}^{r} p_{i} v(a_{i})+ \gamma= \sum_{i=1}^{r} p_{i}v(e_{i})+ \gamma=  \sum_{i=1}^{r} p_{i} \epsilon_{i}+ \gamma.
\end{align}
 Because $\Gamma_{(L,v)}$ and $\Gamma_{(L,\hat{v})}$ are isomorphic,
\begin{align}\label{eq2}
k\hat{v}(l)= \sum_{i=1}^{r} p_{i} \hat{v}(a_{i})+ \gamma= \sum_{i=1}^{r} p_{i} \eta_{i} + \sum_{i=1}^{r} p_{i} \epsilon_{i}+\gamma.
\end{align}
Because $\{ v(a_{i}) \ | \ i \leq r\}$ is a $\mathbb{Q}$-linearly independent set over $\Gamma_{(C,v)}=\Gamma_{(C,\hat{v})}$ and 
\begin{equation*}
\{ v(a_{i}) \ | \ i \leq r\}=\{ v(e_{i}) \ | \ i \leq r \} \subseteq \Gamma_{(M,v)}=\Gamma_{(M,\hat{v})},
\end{equation*}
then $\{ \epsilon_{1}, \dots, \epsilon_{r}\}$ is also a $\mathbb{Q}$-linearly independent set over $\Gamma_{(C,v)} \subset \Gamma_{(M,v)}=\Gamma_{(M,\hat{v})}$. Thus we can extend it to a maximal set of elements in $\Gamma_{M}$ which are  $\mathbb{Q}$-linearly independent set over $\Gamma_{C}$, say  $\{ \epsilon_{1}, \dots, \epsilon_{r}\} \cup \{ \mu_{\alpha} \ | \ \alpha \in \lambda\}$. Hence, we can find indices $\{ \mu_{\alpha_{1}},\dots, \mu_{\alpha_{t}}\}$ such that
\begin{align*}
s\hat{v}(m)= \sum_{i=1}^{r} r_{i} \epsilon_{i}+ \sum_{i=1}^{t}q_{i} \mu_{\alpha_{i}}+ \gamma', \ \text{where $r_{i}, q_{i} \in \mathbb{Z}$, $\gamma' \in \Gamma_{C}$ and $s \in \mathbb{N}^{\geq 1}$.}
\end{align*}
Since $\hat{v}(l)=\hat{v}(m)$, we must have that $s(k\hat{v}(l))=k(s\hat{v}(m))$ thus:
\begin{align}\label{equalzero}
\underbrace{\sum_{i=1}^{r}sp_{i} \eta_{i}}_{ =\delta \in \Delta} + \underbrace{ \sum_{i=1}^{r} (sp_{i}-kr_{i}) \epsilon_{i}- \sum_{i=1}^{t} kq_{i} \mu_{\alpha_{i}}+ (s\gamma-k\gamma')}_{= \beta \in \Gamma_{M}}=0, 
\end{align}
 Because the elements $\{ \eta_{1}, \dots, \eta_{r}\}$ are infinitesimal with respect to $\Gamma_{(M,v)}^{>0}$, the equation \ref{equalzero} is satisfied if and only if $\delta=\displaystyle{\sum_{i=1}^{r}sp_{i} \eta_{i}=0}$ and $\beta= \displaystyle \sum_{i=1}^{r} (sp_{i}-kr_{i}) \epsilon_{i}- \sum_{i=1}^{t} kq_{i} \mu_{\alpha_{i}}+ (s\gamma-k\gamma')=0$. As $0< \eta_{1}<< \eta_{2}<<\dots << \eta_{r}$, then  $\displaystyle{\sum_{i=1}^{r}sp_{i} \eta_{i}=0}$ if and only if $p_{i}=0$ for all $i \leq r$. Then,
\begin{align*}
- \sum_{i=1}^{r} kr_{i} \epsilon_{i}- \sum_{i=1}^{t} kq_{i} \mu_{\alpha_{i}}+ (s\gamma-k\gamma')=0
\end{align*}
but the $\mathbb{Q}$-linear independence of $\{ \epsilon_{1}, \dots, \epsilon_{r}\} \cup \{ \mu_{\alpha_{1}}, \dots, \mu_{\alpha_{t}}\}$ over $\Gamma_{C}$ implies that $r_{i}=0$ for all $i \leq r$ and $q_{i}=0$ for all $i \leq t$. Summarizing we have that $k\hat{v}(l)=\gamma$ and $s\hat{v}(m)=\gamma'$, thus $\gamma$ is $k$-divisible and $\gamma'$ must be $s$-divisible. Hence:
\begin{align*}
\hat{v}(l)=\frac{\gamma}{k}= \frac{\gamma'}{s}=\hat{v}(m).
\end{align*}
As $C$ is a model of $T$, $\Gamma_{C}$ must be definably closed, thus $\frac{\gamma}{k}= \frac{\gamma'}{s} \in \Gamma_{C}$, as required.
\end{proof}
\begin{claim}{$k_{(M,\hat{v})}$ and $k_{(L,\hat{v})}$ are linearly disjoint over $k_{(C,\hat{v})}$.} \end{claim}
\begin{proof}
By the second hypothesis, so $k_{L}$ and $k_{M}$ are algebraically independent over $k_{C}$. Because $C$ is a model, $k_{C}=\cl(k_{C})$ and since the residue field is of characteristic zero, it must be the case that $k_{C} \leq k_{L}$ is a regular extension. By Fact \ref{mago}, $k_{L}$ and $k_{M}$ are linearly disjoint over $k_{C}$. Each of the places $p^{j}$ and $p^{*}$ are the identity map over $k_{M}$ and $k_{L}$, so is their composition
\begin{align*}
p_{0} \circ p_{1} \circ \dots \circ p_{r-1} \circ p^{*}: k_{(C(L,M),v)} \rightarrow \cl(k_{M},k_{L}).
\end{align*}
Because $k_{(M,v)}$ and $k_{(L,v)}$ are linearly disjoint over $k_{(C,v)}$ then  $k_{(M,\hat{v})}$ and $k_{(L,\hat{v})}$ must be linearly disjoint over $k_{(C,\hat{v})}=k_{(C,v)}$.
\end{proof}
\end{proof}
\par \textit{Step $3$: We find the $\mathcal{L}_{val}$-isomorphism $\hat{\sigma}$ extending $\sigma$ which is the identity on $M$.}
\begin{proof}
By Step $2$ there is some valuation $\hat{v}$ over $C(L,M)$ finer than $v$ satisfying the following conditions:
\begin{itemize}
\item $\Gamma_{(M,\hat{v})} \cap \Gamma_{(L,\hat{v})}=\Gamma_{(C,\hat{v})}$,
\item $k_{(M,\hat{v})}$ and $k_{(L,\hat{v})}$ are linearly disjoint over $k_{(C,\hat{v})}$,
\item for any element $x \in M$, we have that $v(x)=\hat{v}(x)$.
\end{itemize}
 Because $C \subseteq M$, then the valuations $v$ and $\hat{v}$ coincide over $C$. In particular $(C, \hat{v})$ is maximal, so by Fact \ref{maximal} $(M,\hat{v})$ has the good separated basis property over $C$. By Proposition \ref{linearlydisjoint} $M$ has the separated basis property over $(L, \hat{v})$, $M$ and $L$ are linearly disjoint over $C$, the value group of $\Gamma_{(C(L,M), \hat{v})}$ is the group generated by $\Gamma_{(L, \hat{v})}$ and $\Gamma_{(M,\hat{v})}$ over $\Gamma_{(C,\hat{v})}$ and the residue field $k_{(C(L,M),\hat{v})}$ is the field generated by $k_{(L,\hat{v})}$ and $k_{(M, \hat{v})}$ over $k_{(C, \hat{v})}$.\\
 We consider the field $C(\sigma(L), \sigma(M))$ with the valuation $\hat{v}$ such that $\sigma: (C(L,M),\hat{v}) \rightarrow (C(\sigma(L),\sigma(M)),\hat{v})$ is an $\mathcal{L}_{\val}$-isomorphism, which fixes $k_{M} \subset Int_{\Gamma_{L}}^{M}\Gamma_{L}$ and $\Gamma_{M}$ by Step $1$. By Proposition \ref{isoextended2}, we can find an $\mathcal{L}_{\val}$-isomorphism $\tau:(C(L,M),\hat{v}) \rightarrow (C(\sigma(L),M), \hat{v})$ which is the identity on $M$ and $\tau\upharpoonright_{(L,\hat{v})}:(L, \hat{v}) \rightarrow (\sigma(L), \hat{v})$. \\
 
 We want to show that $\tau: C(L,M) \rightarrow C(\sigma(L),M)$ induces as well an $\mathcal{L}_{\val}$-isomorphism with the original valuation $v$. Let $x \in C[L,M]$, without loss of generality we may assume that $\displaystyle{x = \sum_{i=1}^{n} l_{i} m_{i}}$ where $(m_{1},\dots,m_{n})$ is a separated basis of $Vect_{L}(m_{1},\dots,m_{n})$ according to the valuation $\hat{v}$, thus 
 \begin{align*}
 \hat{v}(x)= \hat{v} \big( \sum_{i=1}^{n} l_{i}m_{i}\big)=\min\{ \hat{v}(l_{i}m_{i}) \ | \ i \leq n \} = \hat{v}(m_{j}l_{i}) \ \text{for some $j \leq n$}. 
 \end{align*}
  As $\tau:(C(L,M),\hat{v}) \rightarrow (C(\sigma(L),M), \hat{v})$ is an $\mathcal{L}_{\val}$-isomorphism, $(m_{1},\dots,m_{n})$ is a separated basis of $\Vect_{\sigma(L)}(m_{1},\dots,m_{n})$ and 
 \begin{align*}
 \hat{v}(\tau(x))= \hat{v} \big( \sum_{i=1}^{n}\sigma( l_{i})m_{i}\big)=\min\{ \hat{v}(\sigma(l_{i})m_{i}) \ | \ i \leq n \} = \hat{v}(\sigma(l_{j})m_{j})= \tau(\hat{v}(l_{j}m_{j})). 
 \end{align*}
  By Lemma \ref{inf}$(3)$ $(m_{1},\dots,m_{n})$ is also a separated basis of $\Vect_{L}(m_{1},\dots,m_{n})$ and $\Vect_{\sigma(L)}(m_{1},\dots,m_{n})$ with respect to the valuation $v$. Consequently, 
\begin{align*}
v(x)= v(m_{j}l_{j})= v(m_{j})+ v(l_{j})= v(m_{j})+ \sigma(v(l_{j}))=v(m_{j})+ v(\sigma(l_{j}))= v(m_{j}\sigma(l_{j}))=v(\tau(x)).
\end{align*}
As $x \in C[L,M]$ is an arbitrary element, we conclude that the value group of $C(L,M)$ and $C(\sigma(L),M)$ according to the valuation $v$ is $\Gamma_{M}$ and $\tau$ acts as the identity on $\Gamma_{M}$. 
\end{proof}
Hence $\tau$ is also a $\mathcal{L}_{\val}$-isomorphism between the valued field structures $(C(L,M),v)$ and $(C(\sigma(L),M), v)$ which acts as the identity on $M$ and coincides with $\sigma$ on $L$. \\
\textit{Step $4$: We extend the $\mathcal{L}_{\val}$-isomorphism to a $\mathcal{L}$-isomorphism.}
We want to extend the $\mathcal{L}_{\val}$-isomorphism to a $\mathcal{L}$-isomorphism, thus we first want to extend the isomorphism by adding maps $\tau_{n}: (\mathcal{A}_{n})_{C(L,M)} \rightarrow (\mathcal{A}_{n})_{C(\sigma(L),M)}$. \\
Let $a \in C(L,M)$ we say that $\rv(a)$ is \emph{representable with parameters over $\kInt_{\Gamma_{L}}^{M}$} if there are $l_{1},\dots, l_{s} \in L$  and $m_{1},\dots,m_{s} \in M$ such that $\displaystyle{\rv(a)= \big(\sum_{i \leq s} \rv(l_{i})\rv(m_{i})\big)}$ where for each $i \leq s$, $\rv(m_{i}) \in \kInt_{\Gamma_{L}}^{M}$.
\begin{claim}\label{chumbi}{For each $x \in C[L,M]$ there are $ \hat{m}\in M$ and $a \in C(L,M)$, such that $v(a)=0$, $\rv(x)=\rv(a)\rv(\hat{m})$ and $\rv(a)$ is representable with parameters over $\kInt_{\Gamma_{L}}^{M}$. Furthermore, $x= a \hat{m}$ and $\displaystyle{a=\sum_{i \leq n} l_{i}m_{i}}$ for some $l_{i} \in L$ and $m_{i}\in M$, thus $\displaystyle{\tau(x)= \tau(a) \tau(\hat{m})= \big( \sum_{i \leq n} \sigma(l_{i}) m_{i} \big) \hat{m}}$.}\end{claim}
\begin{proof}
Fix an element $x \in C[L,M]$, and suppose $\displaystyle{x= \sum_{i \leq n} l_{i}m_{i}}$. Because $M$ has the separated basis property over $L$ according to the valuation $v$, \begin{equation*}
v(x)=v(\sum_{i \leq n} l_{i}m_{i})=\min\{ v(l_{i}m_{i}) \ | \ i \leq n\}= v(l_{i_{0}}m_{i_{0}}).
\end{equation*}
As $\Gamma_{L} \subseteq \Gamma_{M}$ there is some $\hat{m} \in M$ such that $v(l_{i_{0}}m_{i_{0}})= v(\hat{m})$. Let $I=\{i \leq n \ | \ v(l_{i}m_{i})=v(\hat{m})\}$ and $\displaystyle{a= \frac{x}{\hat{m}}= \sum_{i\leq n} l_{i} \frac{m_{i}}{\hat{m}}}$.\\
For each $ i \in I$, $v(\frac{m_{i}}{\hat{m}})=-v(l_{i})=-\lambda_{i} \in \Gamma_{L}$, thus $\rv\big(\frac{m_{i}}{\hat{m}}\big) \in \kInt_{\Gamma_{L}}^{M}$. Then, $\displaystyle{\rv(a)= \sum_{i \in I} \rv(l_{i}) \rv\big(\frac{m_{i}}{\hat{m}}\big)}$ where each $\rv\big(\frac{m_{i}}{\hat{m}}\big) \in \kInt_{\Gamma_{L}}^{M}$. \\
Summarizing, we have that $x=a \hat{m}$, so $\rv(x)=\rv(a) \rv(\hat{m})$ where $\rv(a)$ is representable with parameters over $\kInt_{\Gamma_{L}}^{M}$. For the second part of the statement, we simply notice that $\displaystyle{\tau(a)= \frac{\tau(x)}{\tau(\hat{m})}= \sum_{i\leq n} \sigma(l_{i}) \frac{m_{i}}{\hat{m}}}$, as required. 
\end{proof}
Thus, given $x_{1}, x_{2}, y_{1}, y_{2} \in C[L,M]$ we can find elements $m_{1}, n_{1}, m_{2}, n_{2} \in M$, $a_{1}, b_{1} ,a_{2}, b_{2} \in \mathcal{O}_{C(L,M)}^{\times}$  such that $x_{1}=a_{1}m_{1}, \ x_{2}=a_{2}m_{2}, \ y_{1}=b_{1}n_{1}$ and $y_{2}=b_{2}n_{2}$.\\
We argue that for each $n \in \mathbb{N}$, if $v\big( \frac{x_{1}}{x_{2}}\big), v\big( \frac{x_{2}}{y_{2}}\big) \in n \Gamma$ and $\rho_{n} \big( \rv\big( \frac{x_{1}}{x_{2}}\big)\big)= \rho_{n} \big( \rv\big( \frac{y_{1}}{y_{2}}\big)\big)$ then 
\begin{equation*}
\rho_{n} \big( \rv\big( \frac{\tau(x_{1})}{\tau(x_{2})}\big)\big)= \rho_{n} \big( \rv\big( \frac{\tau(y_{1})}{\tau(y_{2})}\big)\big).
\end{equation*}
Note that: 
\begin{align*}
\rho_{n} \big(\rv\big( \frac{x_{1}}{x_{2}}\big)\big)&=\rho_{n}\big(\rv \big( \frac{y_{1}}{y_{2}}\big)\big) \  \text{if and only if}  \ \rho_{n} \big( \rv \big(\frac{a_{1}}{b_{1}} \frac{m_{1}}{n_{1}}\big) \big)= \rho_{n} \big( \rv \big(\frac{a_{2}}{b_{2}} \frac{m_{2}}{n_{2}}\big) \big),\\
 \text{if and only if} \ &\rho_{n} (\rv(a_{1})) \rho_{n}(\rv(b_{1} )^{-1} )\rho_{n} \big( \rv \big( \frac{m_{1}}{n_{1}}\big)\big)= \rho_{n}( \rv(a_{2})) \rho_{n}( \rv(b_{2})^{-1}) \rho_{n} \big( \rv \big( \frac{m_{2}}{n_{2}}\big)\big), 
 \end{align*}
 where $\rv(a_{1}), \rv(b_{1}), \rv(a_{2})$ and $\rv(b_{2})$ are representable with parameters in $\kInt_{\Gamma_{L}}^{M}$ and $\rho_{n} \big( \rv \big( \frac{m_{1}}{n_{1}}\big)\big), \rho_{n} \big( \rv \big( \frac{m_{2}}{n_{2}}\big)\big) \in \mathcal{A}_{M}$. In particular, the equality $\rho_{n} \big(\rv\big( \frac{x_{1}}{x_{2}}\big)\big)=\rho_{n}\big(\rv \big( \frac{y_{1}}{y_{2}}\big)\big) $ can be represented by a formula satisfied by $L$ using parameters in $C\kInt_{\Gamma_{L}}^{M}\mathcal{A}_{M}$. As $\sigma: L \rightarrow L^{'}$ is an elementary map fixing $C\kInt_{\Gamma_{L}}^{M}\mathcal{A}_{M}$, the same formula must hold for $\sigma(L)$, thus $\rho_{n} \big( \rv\big( \frac{\tau(x_{1})}{\tau(x_{2})}\big)\big)= \rho_{n} \big( \rv\big( \frac{\tau(y_{1})}{\tau(y_{2})}\big)\big)$. \\
 Hence, for each $n \in \mathbb{N}$, we can naturally define the map $\tau_{n}:(\mathcal{A}_{n})_{C(L,M)}\rightarrow (\mathcal{A}_{n})_{C(\sigma(L),M)}$, where for $x,y \in C[L,M]$ we define  $\tau_{n}\big(\rho_{n} \big( \rv \big( \frac{x}{y}\big)\big)\big)=\rho_{n} \big( \rv \big( \frac{\tau(x)}{\tau(y)}\big)\big)$. Take $\mathbf{t}= \tau \cup \{ \tau_{n} \ | \ n \in \mathbb{N}\}$, then $\mathbf{t}:C(L,M) \rightarrow C(\sigma(L,M))$ is a $\mathcal{L}_{RV}$-isomorphism which satisfies the following conditions:
 \begin{enumerate}
 \item $\mathbf{t}\upharpoonright_{\mathcal{A}_{C(L,M)}}:\mathcal{A}_{C(L,M)} \rightarrow \mathcal{A}_{C(\sigma(L,M))}$ is partial elementary map of $\mathcal{A}_{\mathfrak{C}}$. This follows by Claim \ref{chumbi} combined with the fact that $\sigma: L \rightarrow L^{'}$ is a partial elementary map fixing $C\kInt_{\Gamma_{L}}^{M} \mathcal{A}_{M}$. 
 \item $\mathbf{t}\upharpoonright_{\Gamma_{C(L,M)}}:\Gamma_{C(L,M)}\rightarrow \Gamma_{C(\sigma(L,M))}$ is partial elementary map of $\Gamma_{\mathfrak{C}}$. In fact, $\Gamma_{C(L,M)}= \Gamma_{M}=\Gamma_{C(\sigma(L),M)}$ and $\mathbf{t}$ acts as the identity on the value group. 
 \end{enumerate}
 By quantifier elimination relative to the power residue sorts and the value group, the partial isomorphism $\tau$ must be an elementary map. It coincides with $\sigma$ over $L$ and is the identity on $M$, so we conclude that $\tp(L/M)=\tp(L'/M)$, as desired. 
 \end{proof}
% \textcolor{red}{Incluir enunciado en t\'erminos de dominaci\'on}
% \textcolor{red}{revisar estos remarks}
 
 We restate the result in terms of domination.
 \begin{corollary} Let  $L$ be an elementary substructure of $\mathfrak{C}$ and let $C \subseteq L$ be a maximal model of $T$. Then the type $\tp(L/C)$ is dominated over its value group by the sorts internal to the residue field, that is for any field extension $C\Gamma_{L} \subseteq M$  such that $kInt_{\Gamma_{L}}^{M} \ind_{\Gamma_{L}C}^{qfs}kInt_{\Gamma_{L}}^{L}$ we have $\tp(L/C\Gamma_{L}kInt_{\Gamma_{L}}^{M})\vdash \tp(L/C\Gamma_{L}M)$. 
 \end{corollary}
 \begin{proof}
  Let $C\Gamma_{L} \subseteq M$ such that $kInt_{\Gamma_{L}}^{M} \ind_{C\Gamma_{L}}^{qfs}kInt_{\Gamma_{L}}^{L}$. We aim to prove that $\tp(L/C\Gamma_{L}kInt_{\Gamma_{L}}^{M})\vdash \tp(L/C\Gamma_{L}M)$, this is given an elementary map
 $\sigma: L \rightarrow L'$ fixing $C\Gamma_{L}kInt_{\Gamma_{L}}^{M}$ we can find an automorphism $\tau$  extending $\sigma$ which is the identity on $M$. This is precisely the conclusion of Theorem \ref{iso2}. 
 \end{proof}

\bibliographystyle{plain}

%\end{theorem}
\end{document}